\documentclass[11pt]{amsart}
\usepackage[all]{xy}

\usepackage[utf8]{inputenc}		
\usepackage[T1]{fontenc}
\usepackage[english]{babel}

\usepackage{amsfonts}			
\usepackage{amsmath}
\usepackage{amssymb}
\usepackage{amsthm}
\usepackage{mathtools}
\makeatletter					
\@namedef{subjclassname@2020}{	
	\textup{2020} Mathematics Subject Classification}
\makeatother

\usepackage{graphicx}			
\usepackage{tikz}
	\usetikzlibrary{cd}

\usepackage{hyperref}			
	\hypersetup{colorlinks=false, urlcolor=black, linkcolor=black}

\usepackage{enumitem}			

\usepackage{color}				

\usepackage{thmtools}
\usepackage{thm-restate}

\newcommand{\Z}{\mathbb{Z}}						
\newcommand{\R}{\mathbb{R}}						

\newcommand{\K}{\mathbb{K}}

\renewcommand{\phi}{\varphi}

\renewcommand{\S}{\mathbb{S}}					
\newcommand{\B}{\mathbb{B}}

\newcommand{\eps}{\varepsilon}					

\newcommand{\dd}								
	{\mathop{}\!\mathrm{d}}						
\newcommand{\ddn}[1]							
	{\mathop{}\!\mathrm{d^{#1}}}

\newcommand{\abs}[1]							
	{\left| #1 \right|}
\newcommand{\smallabs}[1]						
	{\lvert #1 \rvert}	
\newcommand{\norm}[1]							
	{\left\lVert #1 \right\rVert}	
\newcommand{\smallnorm}[1]						
	{\lVert #1 \rVert}						
\newcommand{\ip}[2]								
	{\left< #1 , #2 \right>}

\DeclareMathOperator{\vol}{vol}					
\DeclareMathOperator{\spt}{spt}
\DeclareMathOperator{\dist}{dist}					

\DeclareMathOperator{\adm}{adm}

\DeclareMathOperator{\ext}{ext}
\DeclareMathOperator{\diam}{diam}

\DeclareMathOperator{\im}{im}					

\newcommand{\loc}{\mathrm{loc}}					

\newcommand{\cH}{\mathcal{H}}

\newcommand{\cS}{\mathcal{S}}

\newcommand{\moddens}{\operatorname{Mod}}

\def\XXint#1#2#3{{\setbox0=\hbox{$#1{#2#3}{\int}$}
		\vcenter{\hbox{$#2#3$}}\kern-.5\wd0}}

\newtheorem{thm}{Theorem}[section]{\bf}{\it}
\newtheorem{lemma}[thm]{Lemma}
\newtheorem{prop}[thm]{Proposition}
\newtheorem{cor}[thm]{Corollary}

{\bf}{\it}

{\bf}{\it}

\theoremstyle{definition}

\theoremstyle{remark}
\newtheorem{rem}[thm]{Remark}
\numberwithin{equation}{section}

\begin{document}
	
\title[BBM and Sobolev Forms]{A Bourgain-Brezis-Mironescu -type characterization for Sobolev differential forms}
\author{Ilmari Kangasniemi}

\address{Department of Mathematical Sciences, University of Cincinnati, P.O.\ Box 210025, Cincinnati, OH 45221, USA}
\email{kangaski@ucmail.uc.edu}
	
\thanks{This work was partially supported by the NSF grant DMS-2247469.}


\maketitle

\begin{abstract}
	Given a bounded domain $\Omega \subset \R^n$, a result by Bourgain, Brezis, and Mironescu characterizes when a function $f \in L^p(\Omega)$ is in the Sobolev space $W^{1,p}(\Omega)$ based on the limiting behavior of its Besov seminorms. We prove a direct analogue of this result which characterizes when a differential $k$-form $\omega \in L^p(\wedge^k T^* \Omega)$ has a weak exterior derivative $d\omega \in L^p(\wedge^{k+1} T^* \Omega)$, where the analogue of the Besov seminorm that our result uses is based on integration over simplices.
\end{abstract}

\section{Introduction}

Let $\Omega$ be a bounded domain in $\R^n$, let $p \in (1, \infty)$, and let $\theta \in (0, 1)$. The \emph{Besov seminorm} $\norm{f}_{B^{\theta}_{p,p}(\Omega)}$, also known as the \emph{Slobodeckij seminorm}, is defined for measurable functions $f \colon \Omega \to \R$ by
\begin{equation}\label{eq:Besov_def}
	\norm{f}_{B^{\theta}_{p,p}(\Omega)}^p = \int_{\Omega^2} \frac{\abs{f(y) - f(x)}^p}{\abs{x-y}^{n + \theta p}} \, dx \, dy.
\end{equation}
This seminorm leads to the definition of fractional Sobolev spaces $W^{\theta, p}(\Omega)$, also called \emph{Sobolev-Slobodeckij spaces}, which are equipped with the norm $\norm{f}_{L^p(\Omega)} + \norm{f}_{B^{\theta}_{p,p}(\Omega)}$.

If one sets $\theta = 1$ in \eqref{eq:Besov_def}, the resulting seminorm does not match the Sobolev seminorm $\norm{\nabla f}_{L^p(\Omega)}$ for weakly differentiable functions $f$. However, this issue can be rectified with a suitable extra coefficient. Indeed, in \cite[Corollary 2]{Bourgain-Brezis-Mironescu_Characterization}, Bourgain, Brezis, and Mironescu gave the following characterization of Sobo\-lev functions using the Besov seminorm.

\begin{thm}[\textbf{Bourgain--Brezis--Mironescu}]\label{thm:BBM}
	Let $\Omega$ be a bounded smooth domain in $\R^n$, let $p \in (1, \infty)$, and let $f \in L^p(\Omega)$. Then $f \in W^{1,p}(\Omega)$ if and only if $\limsup_{\theta \to 1^{-}} (1-\theta) \norm{f}_{B^{\theta}_{p,p}(\Omega)}^p$ is finite. Moreover, if this is the case, then we have
	\[
		\lim_{\theta \to 1^{-}} (1-\theta) \norm{f}_{B^{\theta}_{p,p}(\Omega)}^p
		= C(n, p) \norm{\nabla f}_{L^p(\Omega)}^p.
	\]
\end{thm}

Note that in the above form, the characterization is not true without any regularity assumptions on the domain; see \cite[Remark 5]{Brezis_BBMSurvey}, where a counterexample is provided in a slit disk. However, with suitable adjustments to the Besov seminorm, it is possible to eliminate this need for regularity assumptions; see the recent works by Drelichman and Dur\'an \cite{DrelichmanDuran-BBMArbitrary} and Mohanta \cite{Mohanta_GeneralBBM}. The characterization has spurred a significant amount of follow-up research, especially in the recent few years; see e.g.\ \cite{Ponce_BBM, Nguyen_BBM-BV-Duke, Leoni-Spector_BBMGen, Leoni-Spector_BBMGen-Corrigendum, BrezisNguyen-BBMRevisited, Pinamonti-Squassina-Vecchi_Magnetic-BBM, Brezis-VanSchaftingen-Yung_BBM-like, Brezis-Seeger-VanSchafingen-Yung_RevisitingSobolev, GarofaloTralli_BBM-Carnot}. Moreover, the characterization has also provided one of the multiple available approaches to studying first order Sobolev theory on metric measure spaces; see e.g.\ \cite{Munnier-BBM_Metric, DiMarinoSquassina-MetricBBM, Gorny_BBM-Metric_EuclTangents, LahtiPinamontiZhou_MetricBBM}.

In this paper, we prove an analogue of Theorem \ref{thm:BBM} for Sobolev differential forms in $\R^n$. To this end, recall that if $\Omega \subset \R^n$ is open and $k \in \left\{0, 1, \dots, n-1\right\}$, a locally integrable measurable differential $k$-form $\omega \in L^1_\loc(\wedge^k T^*\Omega)$ has a \emph{weak exterior derivative} $d\omega \in L^1_\loc(\wedge^{k+1} T^*\Omega)$ if
\[
	\int_\Omega \omega \wedge d\eta = (-1)^{k+1} \int_\Omega d\omega \wedge \eta 
\]
for every compactly supported smooth test form $\eta \in C^\infty_0(\wedge^{n-k-1} T^* \Omega)$. Forms $\omega \in L^p(\wedge^k T^* \Omega)$ with an $L^p$-integrable weak exterior derivative $d\omega \in L^p(\wedge^{k+1} T^* \Omega)$ then form a Sobolev space, denoted $W^{d,p}(\wedge^k T^* \Omega)$. The spaces $W^{d,p}(\wedge^k T^* \Omega)$ generalize e.g.\ spaces of $L^p$-integrable vector fields with an $L^p$-integrable distributional curl or divergence. 

In many applications of Sobolev differential forms, the spaces $W^{d,p}(\wedge^k T^* \Omega)$ exhibit more useful behavior than the spaces $W^{1,p}(\wedge^k T^*\Omega)$ of $k$-forms with coefficients in $W^{1,p}(\Omega)$. This is for instance since if $\omega \in C^\infty_0(\wedge^k T^* \R^m)$ is a smooth compactly supported $k$-form, $f \in W^{1,p}(\Omega, \R^m)$ is Sobolev map with $p \ge k+1$, and $\Omega\subset \R^n$ is a bounded domain, then the pull-back $f^* \omega$ of $\omega$ is in $W^{d,p/(k+1)}(\wedge^k T^* \Omega)$, but is not necessarily in any space $W^{1, q}_\loc(\wedge^k T^* \Omega)$. For further exposition on the spaces $W^{d,p}(\wedge^k T^* \Omega)$, we refer readers to e.g.\ \cite{Iwaniec-Lutoborski, Iwaniec-Scott-Stroffolini}

\subsection{The main result in bounded convex domains}

Let $\Omega \subset \R^n$ be an open domain. We let $M^k(\Omega; \R)$ denote the space of \emph{measurable real-valued $k$-multifunctions} on $\Omega$; that is, an element $F \in M^k(\Omega; \R)$ is a measurable function $F \colon \Omega^{k+1} \to \R$. We then recall the \emph{Alexander-Spanier differential} $d \colon M^k(\Omega; \R) \to M^{k+1}(\Omega; \R)$ on multifunctions, which is defined for $F \in M^k(\Omega; \R)$ by
\begin{equation}\label{eq:AS_differential}
	(dF)(x_0, x_1, \dots, x_{k+1}) = \sum_{i = 0}^k (-1)^i F(x_0, x_1, \dots, x_{i-1}, x_{i+1}, \dots, x_{k+1})
\end{equation}

Next, if $\Omega \subset \R^n$ is convex and $\omega \in L^p(\wedge^k T^* \Omega)$, then we can redefine $\omega$ in a null-set to ensure that it is Borel measurable, and afterwards define an integration function
\[
	I_\omega \colon \Omega^{k+1} \to \R, \quad I_\omega(x_0, \dots, x_k) = \int_{\Delta(x_0, \dots, x_k)} \omega,
\]
where $\Delta(x_0, \dots, x_k)$ is the oriented $k$-simplex with corners at $x_i$. As we verify in Section \ref{sect:integration_function}, $I_\omega$ is well defined and finite-valued outside a null-set of $\Omega^{k+1}$, $I_\omega$ is Borel measurable, and changing $\omega$ in a null-set only changes $I_\omega$ in a null-set. The definition of $I_\omega$ can then be extended to non-convex domains $\Omega$ by zero extending $\omega$ outside $\Omega$. We observe that the Alexander-Spanier -differential of the integration function $I_\omega$ corresponds to integrating $\omega$ over the boundary of $\Delta(x_0, \dots, x_{k+1})$; that is, 
\[
	dI_\omega(x_0, \dots, x_{k+1}) = \int_{\partial \Delta(x_0, \dots, x_{k+1})} \omega, \quad x_0, \dots, x_{k+1} \in \Omega.
\]

With these definitions, we may state the basic form of our main result, which proves a direct analogue for Theorem \ref{thm:BBM} for the spaces $W^{d,p}(\wedge^k T^* \Omega)$ when $\Omega$ is convex. We discuss versions of this result without the convexity assumption later in the introduction.

\begin{thm}\label{thm:BBM_for_forms_convex}
	Let $\Omega \subset \R^n$ be an open, bounded, convex domain, let $p, \in (1, \infty)$, and let $\omega \in L^p(\wedge^k T^* \Omega)$. Then the following conditions are equivalent.
	\begin{enumerate}[label=(\roman*)]
		\item $\omega \in W^{d,p}(\wedge^k T^*\Omega)$; i.e., $\omega$ has an $L^p$-integrable weak exterior derivative $d\omega \in L^p(\wedge^{k+1} T^* \Omega)$.
		\item The integration function $I_\omega$ of $\omega$ satisfies
		\begin{equation}\label{eq:BBM_for_forms_limsup}
			\limsup_{\theta \to 1^{-}} \int_{\Omega^{k+2}} \frac{(1 - \theta)^{k+1} \abs{dI_\omega(x_0, \dots, x_{k+1})}^pdx_0 \ldots dx_{k+1}}{\left(\abs{x_1 - x_0} \abs{x_2 - x_0} \cdots \abs{x_{k+1} - x_0}\right)^{n + \theta p}}
			< \infty.
		\end{equation}
		\item The integration function $I_\omega$ of $\omega$ satisfies
		\begin{equation}\label{eq:BBM_for_forms_limit}
			\lim_{\theta \to 1^{-}} \int_{\Omega^{k+2}} \frac{(1 - \theta)^{k+1} \abs{dI_\omega(x_0, \dots, x_{k+1})}^pdx_0 \ldots dx_{k+1}}{\left(\abs{x_1 - x_0} \abs{x_2 - x_0} \cdots \abs{x_{k+1} - x_0}\right)^{n + \theta p}}
			< \infty;
		\end{equation}
		i.e.\ the limit exists and is finite.
	\end{enumerate}
	Moreover, if the above conditions hold, then the limit in \eqref{eq:BBM_for_forms_limit} is uniformly comparable with $\norm{d\omega}_{L^p(\Omega)}^p$, with comparison constants depending only on $p$, $n$, and $k$. More precisely, there exists a $K = K(p, k)$ for which the limit in \eqref{eq:BBM_for_forms_limit} equals $K \bigl\lVert\abs{d\omega}_{\S,p}\bigr\rVert_{L^p(\Omega)}^p$,
	where the norm $\abs{\alpha}_{\S,p}$ for $l$-covectors $\alpha \in \operatorname{Alt}_l(V)$ on an $n$-dimensional normed space $V$ is defined by
	\begin{equation}\label{eq:alt_pointwise_norm}
		\abs{\alpha}_{\S,p}^p = \int_{\{\abs{v_1} = \ldots = \abs{v_l} = 1\}} \abs{\alpha(v_1, \dots, v_l)}^p \, d\cH^{n-1}(v_1) \dots d\cH^{n-1}(v_l).
	\end{equation}
\end{thm}

Note that outside the extra assumption of convexity, the case $k = 0$ of this result matches the original characterization of Bourgain, Brezis, and Mironescu that was stated in Theorem \ref{thm:BBM}. Indeed, a $0$-form is a function $f \colon \Omega \to \R$. For such a function $f$, we have $I_{f}(x) = f(x)$, and thus $dI_{f}(x, y) = f(y) - f(x)$. Hence, \eqref{eq:BBM_for_forms_limit} equals $\lim_{\theta \to 1^{-}} (1-\theta) \norm{f}_{B^{\theta}_{p,p}(\Omega)}^p$ for $f$. 

We also note that for $1$-covectors $\alpha$ on $\R^n$, the norm $\abs{\alpha}_{\S,p}$ given in Theorem \ref{thm:equiv_of_norms_general} is a constant multiple of the usual Euclidean norm $\abs{\alpha}$. Indeed, if $\alpha$ is nonzero, then one can write $\abs{\alpha}_{\S,p} = \abs{\alpha} \cdot \bigl\lvert\abs{\alpha}^{-1}\alpha\bigr\rvert_{\S,p}$, where the quantity $\bigl\lvert\abs{\alpha}^{-1}\alpha\bigr\rvert_{\S,p}$ is independent of $\alpha$ by the rotational symmetry of $\S^{n-1}$. For $k > 1$, this argument is no longer available due to the existence of non-simple $k$-covectors $\alpha \in \operatorname{Alt}_k(V)$ which cannot be written in the form $\alpha = \alpha_1 \wedge \dots \wedge \alpha_k$ with $\alpha_i \in V^*$.

\subsection{$L^p$-theory of multifunctions.}

The quantity in \eqref{eq:BBM_for_forms_limsup} can be used to define an $L^p$-seminorm on $k$-multifunctions. Namely, if $F \in M^k(\Omega; \R)$ and if $E \subset \Omega$ is measurable, we define
\begin{equation}\label{eq:Lp_for_multifunct_def}
	\norm{F}_{L^p M^k(E)}^p\\ 
	= \limsup_{\theta \to 1^{-}} \int_{E^{k+1}} \frac{(1 - \theta)^k \abs{F(x_0, \dots, x_k)}^pdx_0 \ldots dx_k}{\left(\abs{x_1 - x_0} \abs{x_2 - x_0} \cdots \abs{x_k - x_0}\right)^{n + \theta p}}.
\end{equation}
It follows that $\norm{\cdot}_{L^p M^k(E)}$ defines a $[0, \infty]$-valued seminorm on $M^k(\Omega; \R)$, and multifunctions $F \in M^k(\Omega; \R)$ with $\norm{F}_{L^p M^k(E)} < \infty$ can therefore be thought of as being $L^p$-integrable over $E$.

By Theorem \ref{thm:BBM_for_forms_convex}, if $\omega \in W^{d,p}(\wedge^k T^* \Omega)$ on a convex $\Omega$, then $\norm{dI_\omega}_{L^p M^{k+1}(\Omega)}$ and $\norm{d\omega}_{L^p(\Omega)}$ are uniformly comparable. This is also true for $\norm{I_\omega}_{L^p M^k(\Omega)}$ and $\norm{\omega}_{L^p(\Omega)}$, as we point out in the following theorem.

\begin{thm}\label{thm:equiv_of_norms_basic}
	Let $\Omega \subset \R^n$ be open and bounded, let $p \in [1, \infty)$, $k \in \{0, \dots, n\}$, and let $\omega \in L^p(\wedge^k T^* \Omega)$, where we extend $\omega$ to $L^p(\wedge^k T^* \R^n)$ by setting $\omega = 0$ outside $\Omega$. Then, 
	\[
		\norm{I_\omega}_{L^p M^k(\Omega)} 
		= \lim_{\theta \to 1^{-}} \left( \int_{\Omega^{k+1}} \frac{(1 - \theta)^k \abs{F(x_0, \dots, x_k)}^pdx_0 \ldots dx_k}{\left(\abs{x_1 - x_0} \abs{x_2 - x_0} \cdots \abs{x_k - x_0}\right)^{n + \theta p}} \right)^\frac{1}{p}.
	\]
	Moreover, there exists a constant $K = K(n, p)$ such that $\norm{I_\omega}_{L^p M^k(\Omega)} = K \bigl\lVert\abs{\omega}_{\S,p}\bigr\rVert_{L^p(\Omega)}$, where $\abs{\alpha}_{\S,p}$ is as in \eqref{eq:alt_pointwise_norm}. In particular, $\norm{I_\omega}_{L^p M^k(\Omega)}$ is uniformly comparable with $\norm{\omega}_{L^p(\Omega)}$, with comparison constants depending only on $n$, $p$, and $k$.
\end{thm}

We note that just like differential forms act as the basis of de Rham cohomology, multifunctions act as the basis for the classical theory of Ale\-xan\-der-Spanier cohomology. Recall that, given a topological space $X$, a ring of coefficients $\K$, and $k \in \Z$, the space $C^k_{\text{AS}}(X; \K)$ of $\K$-valued \emph{Alexander-Spanier $k$-cochains} is the quotient of the space of all $k$-multifunctions $F \colon X^{k+1} \to \K$ with an equivalence relation $\sim$, where $F \sim G$ if $F\vert_U = G\vert_U$ for some neighborhood $U \subset X^{k+1}$ of the diagonal $\Delta_k(X) = \{(x, \dots, x) : x \in X\} \subset X^{k+1}$. The differential $d$ from \eqref{eq:AS_differential} then descends to the quotient space, yielding the Alexander-Spanier complex
\[
\dots \xrightarrow{d} C^{k-1}_{\text{AS}}(X; \K)
\xrightarrow{d} C^{k}_{\text{AS}}(X; \K)
\xrightarrow{d} C^{k+1}_{\text{AS}}(X; \K)
\xrightarrow{d} \dots
\]
The cohomology spaces $\ker d / \im d$ of this complex form the \emph{Alexander-Spanier cohomology} of the space $X$; see e.g.~\cite{Spanier_AlgTopoBook} or~\cite{Massey_ASBook} for details. Through use of sheaf theory, it can be shown that Alexander-Spanier cohomology often coincides with other cohomology theories, such as the commonly used singular cohomology, or in the case $\K = \R$, the de Rham cohomology of differential forms; for details, see e.g.\ \cite[Sections I.7, III.2]{Bredon_book}.

Notably, there is a sense of similarity between the equivalence relation $\sim$ used in the definition of Alexander-Spanier cohomology, and the equivalence relation used to construct the associated normed space of the seminorm \eqref{eq:Lp_for_multifunct_def}. This similarity is most clear in the case $k = 1$; for instance, by Remark \ref{rem:diagonal behavior}, if $\Omega$ is a bounded domain, $F, G \in M^1(\Omega; \R)$ with $\norm{F}_{L^p M^1(\Omega)} < \infty$ and $\norm{G}_{L^p M^1(\Omega)} < \infty$, and $F\vert_U = G\vert_U$ for some uniform neighborhood $U$ of the diagonal $\Delta_1(\Omega) \subset \Omega^2$, then $\norm{F - G}_{L^p M^1(\Omega)} = 0$. For $k > 1$, a simple rigorous observation like this becomes more elusive, but since Theorem \ref{thm:BBM_for_forms_convex} indicate that this does hold for $k$-multifunctions of the form $I_\omega$ with $\omega \in L^p(\wedge^k T^* \Omega)$, there could potentially be a larger natural class of measurable $k$-multifunctions where behavior of this type holds.

Previously, $L^p$-norms for multifunctions, including their associated Ale\-xan\-der-Spanier complex, have been studied by e.g.\ Bartholdi, Schick, Smale and Smale \cite{Bartholdi-Schick-Smale-Smale_AS-Hodge-theory}, Genton \cite{Genton_thesis-Lp-Alexander-Spanier}, and Hinz and Kommer \cite{Hinz-Kommer_non-local-AS}. These works generally rely on restricting oneself to fixed scales to achieve the desired topological properties of the complex. The works \cite{Bartholdi-Schick-Smale-Smale_AS-Hodge-theory, Genton_thesis-Lp-Alexander-Spanier} focus on norms that are uniformly comparable to a standard $L^p$-norm on $\Omega^{k+1}$, but \cite{Hinz-Kommer_non-local-AS} considers $L^p$-norms on $\Omega^{k+1}$ involving unbounded kernels. The seminorm \eqref{eq:Lp_for_multifunct_def} is partially inspired by the strategies of using iterated kernels in \cite{Bartholdi-Schick-Smale-Smale_AS-Hodge-theory, Hinz-Kommer_non-local-AS}, with the main difference being the additional limiting process of kernels.

A notable property of the seminorm $\norm{\cdot}^p_{L^p M^k(\Omega)}$, along with its variants that we introduce later, is that it can be expressed using purely metric data. Thus, there is potential for a seminorm like it to act as a basis for an approach to studying Sobolev-regular differential forms in the metric setting, though there remain significant challenges in building such a theory, such as finding a metric approach to specifying which multifunctions arise from integrating a differential form. Current approaches to differential forms in the metric setting include the non-smooth calculus on RCD spaces developed by Gigli, see e.g.\ \cite{Gigli-Pasqualetto_RCD-DG-book}, and the polylipschitz forms of Pankka and Soultanis \cite{Pankka-Soultanis_Polylipschitz2} which act as a pre-dual to Ambrosio-Kirchheim currents \cite{Ambrosio-Kirchheim_currents-acta} and most closely correspond to the space $W^{d, \infty}(\wedge^k T^* \Omega)$.

\subsection{Versions of Theorem \ref{thm:BBM_for_forms_convex} for non-convex domains.}

Without making any adjustments to the domain of integration in \eqref{eq:Besov_def}, a typical level of generality for more recent formulations of Theorem \ref{thm:BBM} is in domains allowing for extension of Sobolev maps. In our case, we are similarly able to generalize \ref{thm:BBM_for_forms_convex} to such a setting without modifying the seminorm $\norm{\cdot}_{L^p M^k(\Omega)}$, but this result is less direct than the one given in Theorem \ref{thm:BBM_for_forms_convex}. 

For the statement, we say that a bounded domain $\Omega \subset \R^n$ is a \emph{(weak) $W^{d,p}(\wedge^k T^* \Omega)$-extension domain} if for every $\omega \in W^{d,p}(\wedge^k T^* \Omega)$, there exists a $\omega_{\ext} \in W^{d,p}(\wedge^k T^* \R^n)$ such that the restriction $\omega_{\ext}\vert_{\Omega}$ of $\omega_{\ext}$ to $\Omega$ is $\omega$. By an argument based on reflection across the boundary, it follows that bounded Lipschitz domains $\Omega \subset \R^n$ are $W^{d,p}(\wedge^k T^* \Omega)$-extension domains; see e.g.\ the discussion in \cite[p. 48]{Iwaniec-Scott-Stroffolini}.

\begin{restatable}{thm}{BBMFormsNonconv}\label{thm:BBM_for_Sobolev_forms_nonconvex}
	Let $\Omega \subset \R^n$ be a bounded domain contained in an open ball $B \subset \R^n$, let $p \in (1, \infty)$, and let $\omega \in L^p(\wedge^k T^* \Omega)$ be a Borel measurable differential $k$-form. Then the following results hold.  
	\begin{enumerate}[label=(\roman*)]
		\item \label{enum:BBM_to_Sobolev} If 
		\[
			\liminf_{\theta \to 1^{-}} \int_{\Omega^{k+2}} \frac{(1 - \theta)^{k+1} \abs{dI_\omega(x_0, \dots, x_{k+1})}^pdx_0 \ldots dx_{k+1}}{\left(\abs{x_1 - x_0} \abs{x_2 - x_0} \cdots \abs{x_{k+1} - x_0}\right)^{n + \theta p}}
			< \infty,
		\]
		then we have $\omega \in W^{d,p}(\wedge^k T^* \Omega)$.
		\item \label{enum:Sobolev_to_BBM} If $\Omega$ is a weak $W^{d,p}(\wedge^k T^* \Omega)$-extension domain and $\omega \in W^{d,p}(\wedge^k T^* \Omega)$, then we have $\norm{dI_{\omega_{\ext}}}_{L^p M^{k+1}(\Omega)} < \infty$, where $\omega_{\ext} \in W^{d,p}(\wedge^k T^* B)$ is a $W^{d,p}$-extension of $\omega$ to $B$.
	\end{enumerate}
\end{restatable}

The statement of Theorem \ref{thm:BBM_for_Sobolev_forms_nonconvex} matches Theorem \ref{thm:BBM_for_forms_convex} if the seminorm $\norm{dI_\omega}_{L^p M^k(\Omega)}$ defined using the zero extension of $\omega$ agrees with the seminorm $\norm{dI_{\omega_{\ext}}}_{L^p M^k(\Omega)}$ defined using a $W^{d,p}$-extension of $\omega$. This is true for functions $f \in W^{1,p}(\Omega)$, since if $(x, y) \in \Omega^2$, then $dI_f(x,y) = f(y) - f(x)$ depends only on the values of $f$ on $\Omega$. However, for higher order forms, a simplex with corners in $\Omega$ can have its boundary exit $\Omega$, and the question thus becomes more delicate. This difference is the main reason why convexity matters in the cases $k > 0$, but does not matter in the classical case $k = 0$.

However, there is a way to get an exact analogue of Theorem \ref{thm:BBM_for_forms_convex} on non-convex domains, which is by adjusting the domain of integration in the seminorm $\norm{\cdot}_{L^p M^k(\Omega)}$. First, a common method to eliminate the boundedness assumption in Theorem \ref{thm:BBM} is to replace the set of integration $\Omega^2$ in \eqref{eq:Besov_def} with $\{(x,y) \in \Omega^2 : \abs{x - y} \le R\}$ for some $R > 0$. Moreover, in \cite{DrelichmanDuran-BBMArbitrary}, Drelichman and Dur\'an generalized Theorem \ref{thm:BBM} to general bounded open $\Omega \subset \R^n$, by replacing the set of integration $\Omega^2$ in \eqref{eq:Besov_def} with $\{(x,y) \in \Omega^2 : \abs{x - y} < c \dist(x, \partial \Omega)\}$ for some $c \in (0, 1)$, and Mohanta \cite{Mohanta_GeneralBBM} further managed to eliminate the boundedness assumption in their result. With higher order analogues to these methods, we are able to obtain a direct generalization of Theorem \ref{thm:BBM_for_forms_convex} for unbounded non-convex domains.

To state this generalization, we define three variants of the seminorm $\norm{\cdot}_{L^p M^k(E)}$. For the first variant, if $E \subset \Omega$ is measurable, $k \in \Z_{\ge 0}$, and $R > 0$, we denote
\[
	\B(E, k, R) = \bigl\{(x_0, \dots, x_k) \in E^{k+1} : \max_{i = 1, \dots, k} \abs{x_0 - x_k} < R\bigr\}.
\]
Then, for $F \in M^k(\Omega; \R)$, we define that
\begin{multline}\label{eq:Lp_for_multifunct_def_loc}
	\norm{F}_{L^p M^k(E\mid R)}^p\\
	= \limsup_{\theta \to 1^{-}} \int_{\B(E, k, R)} \frac{(1 - \theta)^k \abs{F(x_0, \dots, x_k)}^p dx_0 \ldots dx_k}{\left(\abs{x_1 - x_0} \abs{x_2 - x_0} \cdots \abs{x_k - x_0}\right)^{n + \theta p}}.
\end{multline}
For the second variant, for every $c > 0$ and every measurable $E \subset \Omega$, we define a set
\[
	E(k, c) := \bigl\{ (x_0, \dots, x_k) \in E^{k+1} : \max_{i = 1, \dots, k} \abs{x_0 - x_k} < c \dist(x_0, \partial E) \bigr\}.
\]
We then obtain an alternate localized seminorm by
\begin{equation}\label{eq:Lp_for_multifunct_def_alt}
	\norm{F}_{L^p M^{k, c}(E)}^p\\ 
	= \limsup_{\theta \to 1^{-}} \int_{E(k, c)} \frac{(1 - \theta)^k \abs{F(x_0, \dots, x_k)}^p dx_0 \ldots dx_k}{\left(\abs{x_1 - x_0} \abs{x_2 - x_0} \cdots \abs{x_k - x_0}\right)^{n + \theta p}}.
\end{equation}
The final variant performs both localizations at the same time: if $R > 0$ and $c > 0$, then we denote
\begin{multline}\label{eq:Lp_for_multifunct_def_alt_loc}
	\norm{F}_{L^p M^{k, c}(E\mid R)}^p\\
	= \limsup_{\theta \to 1^{-}} \int_{\B(E, k, R) \cap E(k, c)} \frac{(1 - \theta)^k \abs{F(x_0, \dots, x_k)}^p dx_0 \ldots dx_k}{\left(\abs{x_1 - x_0} \abs{x_2 - x_0} \cdots \abs{x_k - x_0}\right)^{n + \theta p}}.
\end{multline}

For the sake of notational convenience, for all of the seminorms we defined above in \eqref{eq:Lp_for_multifunct_def}-\eqref{eq:Lp_for_multifunct_def_alt_loc}, we define counterparts $[\cdot]_{L^p M^k(E)}$, $[\cdot]_{L^p M^k(E\mid R)}$, $[\cdot]_{L^p M^{k, c}(E)}$, and $[\cdot]_{L^p M^{k, c}(E\mid R)}$, where the $\limsup$ in the definition is replaced with a $\liminf$. For instance,
\begin{equation}\label{eq:Lp_for_multifunct_def_liminf}
	[F]_{L^p M^k(E)}^p
	= \liminf_{\theta \to 1^{-}} \int_{E^{k+1}} \frac{(1 - \theta)^k \abs{F(x_0, \dots, x_k)}^p dx_0 \ldots dx_k}{\left(\abs{x_1 - x_0} \abs{x_2 - x_0} \cdots \abs{x_k - x_0}\right)^{n + \theta p}}.
\end{equation}
Moreover, if $F \in M^k(\Omega; \R)$, we say that $\norm{F}_{L^p M^k(E)}$ is \emph{obtained as a limit} if $[F]_{L^p M^k(E)} = \norm{F}_{L^p M^k(E)}$, and similarly for the other seminorms \eqref{eq:Lp_for_multifunct_def_loc}-\eqref{eq:Lp_for_multifunct_def_alt_loc}.

With this notation, the full technical version of Theorem \ref{thm:BBM_for_forms_convex} that also includes non-convex cases is as follows.

\begin{thm}\label{thm:BBM_for_Sobolev_forms_all_cases}
	Let $\Omega \subset \R^n$, $R > 0$, $c \in (0, 1]$, $p, \in (1, \infty)$, $k \in \{0, \dots, n-1\}$, and let $\omega \in L^p(\wedge^k T^* \Omega)$, where we extend $\omega$ to $L^p(\wedge^k T^* \R^n)$ by setting $\omega = 0$ outside $\Omega$. Let $\norm{\cdot}_{p,k+1}$ be one of the seminorms \eqref{eq:Lp_for_multifunct_def}-\eqref{eq:Lp_for_multifunct_def_alt_loc} over $\Omega$, and let $[\cdot]_{p,k+1}$ be the corresponding $\liminf$ -version. Moreover, we assume that
	\begin{enumerate}
		\item if $\norm{\cdot}_{p,k+1} = \norm{\cdot}_{L^p M^{k+1}(\Omega)}$, then $\Omega$ is open, bounded, and convex;\label{enum:BBM_case_bdd_conv}
		\item if $\norm{\cdot}_{p,k+1} = \norm{\cdot}_{L^p M^{k+1}(\Omega\mid R)}$, then $\Omega$ is open and convex;\label{enum:BBM_case_conv}
		\item if $\norm{\cdot}_{p,k+1} = \norm{\cdot}_{L^p M^{k+1,c}(\Omega)}$, then $\Omega$ is open and bounded.\label{enum:BBM_case_bdd}
		\item if $\norm{\cdot}_{p,k+1} = \norm{\cdot}_{L^p M^{k+1,c}(\Omega\mid R)}$, then $\Omega$ is open.\label{enum:BBM_case_unbdd}
	\end{enumerate}
	Then the following conditions are equivalent:
	\begin{enumerate}[label=(\roman*)]
		\item $\omega \in W^{d,p}(\wedge^{k} T^*\Omega)$; \label{enum:BBM_sob}
		\item $\norm{dI_\omega}_{p, k+1} < \infty$; \label{enum:BBM_sup}
		\item $[dI_\omega]_{p,k+1} < \infty$. \label{enum:BBM_inf}
	\end{enumerate}
	Moreover, if $\omega \in W^{d,p}(\wedge^{k} T^*\Omega)$, then we have $[dI_\omega]_{p, k+1} = \norm{dI_\omega}_{p, k+1} = K \bigl\lVert\abs{d\omega}_{\S,p}\bigr\rVert_{L^p(\Omega)}$,
	where $K = K(p, n, k)$ and $\abs{\cdot}_{\S,p}$ are as in Theorem \ref{thm:BBM_for_forms_convex}.
\end{thm}

We also obtain a similar more general analogue to Theorem \ref{thm:equiv_of_norms_basic}.

\begin{thm}\label{thm:equiv_of_norms_general}
	Let $\Omega \subset \R^n$ be open, let $R > 0$, $c > 0$, $p \in [1, \infty)$, $k \in \{0, \dots, n\}$, and let $\omega \in L^p(\wedge^k T^* \Omega)$, where we extend $\omega$ to $L^p(\wedge^k T^* \R^n)$ by setting $\omega = 0$ outside $\Omega$. Let $\norm{\cdot}_{p,k}$ be one of the seminorms \eqref{eq:Lp_for_multifunct_def}-\eqref{eq:Lp_for_multifunct_def_alt_loc} over $\Omega$, and let $[\cdot]_{p,k}$ be the corresponding $\liminf$ -version. Moreover, if $\norm{\cdot}_{p,k} = \norm{\cdot}_{L^p M^k(\Omega)}$ or $\norm{\cdot}_{p,k} = \norm{\cdot}_{L^p M^{k,c}(\Omega)}$, we additionally assume that $\Omega$ is bounded. Then $[I_\omega]_{p,k} = \norm{I_\omega}_{p,k} = K \bigl\lVert\abs{\omega}_{\S,p}\bigr\rVert_{L^p(\Omega)}$, where $K = K(p, n, k)$ and $\abs{\cdot}_{\S,p}$ are as in Theorem \ref{thm:equiv_of_norms_general}. 
\end{thm}

We note that when $k = 0$, the case \eqref{enum:BBM_case_bdd} of Theorem \ref{thm:BBM_for_Sobolev_forms_all_cases} matches the result shown by Drelichman and Dur\'an in \cite{DrelichmanDuran-BBMArbitrary}, with a miniscule improvement of also allowing for $c \in (0, 1]$ instead of $c \in (0, 1)$. However, for the unbounded version of this result shown in case \ref{enum:BBM_case_unbdd}, there is a discrepancy between the version for general $k$ that we show and the version for $k = 0$ which follows from the work of Mohanta \cite{Mohanta_GeneralBBM}. This is because if $k = 0$ and $f \in L^p(\Omega)$, then by the argument given in \cite[Proposition 17]{Mohanta_GeneralBBM}, one always has $\norm{dI_f}_{L^p M^{1,c}(\Omega)} = \norm{dI_f}_{L^p M^{1,c}(\Omega\mid R)}$. However, for $k > 0$, this argument runs into issues, and it is not clear whether one has $\norm{dI_\omega}_{L^p M^{k+1,c}(\Omega)} = \norm{dI_\omega}_{L^p M^{k+1,c}(\Omega\mid R)}$ for $\omega \in L^p(\wedge^k T^* \Omega)$.

\subsection*{Acknowledgments}

The author thanks Josh Kline, Panu Lahti, Pekka Pankka, and Nages\-wari Shanmugalingam for several useful discussions on the topic.

\section{Preliminaries}

We use $C(a_1, \dots, a_k)$ to denote a positive constant that depends on the parameters $a_i$; the value of $C$ can change in each estimate even if the parameters remain the same. We also use the shorthand $A_1 \lesssim_{a_1, \dots, a_k} A_2$ for $A_1 \le C(a_1, \dots, a_k) A_2$. For functions with multiple variables, we use $\hat{x_i}$ to denote an omitted variable: for example $F(x_1, \dots, \hat{x_3}, \dots, x_5) = F(x_1, x_2, x_4, x_5)$.

We use $\B^n(x, r)$ to denote an open ball in $\R^n$ centered at $x \in \R^n$ with radius $r$. The unit ball has the abbreviation $\B^n := \B^n(0, 1)$. Moreover, if $E \subset \R^n$ is a set, we use $\B^n(E, r)$ to denote the set of points $x \in \R^n$ with $\dist (x, E) < \eps$. We let $m_n$ and $\cH^{n}$ denote the $n$-dimensional Lebesgue and Hausdorff measure, respectively. We use $\vol_n$ to denote the volume form in $\R^n$. At times, we write $dx$ instead of $dm_n(x)$ for brevity, where $x \in \R^n$ is a variable of integration; this notation is not to be confused with the exterior derivative of a differential form, as here $x$ cannot be interpreted as a 0-form when $n > 1$.

\subsection{$L^p$ and Sobolev differential forms}

We begin by recalling the basics of Sobolev differential forms, restricting ourselves to the Euclidean setting. As stated in the introduction, for further information on the topic, we refer the reader to e.g.\ \cite{Iwaniec-Lutoborski, Iwaniec-Scott-Stroffolini}.

Given an open set $\Omega \subset \R^n$ and $k \in \{0, \dots, n\}$, a differential $k$-form on $\Omega$ is a section $\omega \mapsto \wedge^k T^*\Omega$ of the $k$:th exterior bundle of $\Omega$. Since we are in the Euclidean setting, any differential form $\omega$ on $\Omega$ can be uniquely written as
\[
	\omega = \sum_{I} \omega_I dx_{I_1} \wedge \dots \wedge dx_{I_k},
\]
where $I = (I_1, \dots, I_k)$ range over $k$-tuples of elements in $\{1, \dots, n\}$ with $I_1 < I_2 < \dots < I_k$, and $\omega_I \colon \Omega \to \R$ are real-valued functions. Recall that differential $k$-forms have a point-wise norm, which in the Euclidean setting can be stated as just
\[
	\abs{\omega_x} = \sqrt{\sum\nolimits _I \abs{\omega_I(x)}^2}
\]
for all $x \in \Omega$.

We say that a differential form $\omega$ is \emph{(Lebesgue) measurable} if all the coefficient functions $\omega_I$ are (Lebesgue) measurable functions. Similarly we define Borel, continuous, and $C^l$-smooth differential forms $\omega$ by requiring that all $\omega_I$ are Borel, continuous, or $C^l$-smooth, respectively. We denote the spaces of continuous, and $C^l$-smooth differential $k$-forms on $\Omega$ by $C(\wedge^k T^* \Omega)$ and $C^l(\wedge^k T^* \Omega)$, respectively. We also define compactly supported counterparts, which we denote by $C_0(\wedge^k T^* \Omega)$ and $C^l_0(\wedge^k T^* \Omega)$, respectively. 

Given $p \in [1, \infty]$, we let $L^p(\wedge^k T^* \Omega)$ denote the space of equivalence classes of measurable differential forms $\omega$ such that $\abs{\omega} \in L^p(\Omega)$. As usual, we consider two $k$-forms in $L^p(\wedge^k T^* \Omega)$ equal if they agree outside a null-set of $\Omega$, and equip the space $L^p(\wedge^k T^* \Omega)$ with the norm $\norm{\omega}_{L^p(\Omega)} = \norm{\abs{\omega}}_{L^p(\Omega)}$. We also use $L^p_\loc(\wedge^k T^* \Omega)$ to denote the space of measurable $k$-forms $\omega$ with $\abs{\omega} \in L^p_\loc(\Omega)$.   

Moreover, as stated in the introduction, we use $W^{d,p}(\wedge^k T^* \Omega)$ to denote the space of $\omega \in L^p (\wedge^k T^* \Omega)$ which have a weak exterior derivative $d\omega \in L^p (\wedge^{k+1} T^* \Omega)$. Here, the \emph{weak exterior derivative} of a $k$-form $\omega \in L^1_\loc(\wedge^k T^* \Omega)$ is a $(k+1)$-form $d\omega \in L^1_\loc(\wedge^{k+1} T^* \Omega)$ such that
\[
	\int_\Omega d\omega \wedge \eta = (-1)^{k+1} \int_{\Omega} \omega \wedge d\eta
\]
for all test forms $\eta \in C^\infty_0(\wedge^{n-k-1} T^* \Omega)$. Note that a weak exterior derivative $d\omega$ is unique up to a null-set of $\Omega$. The spaces $W^{d,p}(\wedge^k T^* \Omega)$ are equipped with the norm $\norm{\omega}_{W^{d,p}(\Omega)} = \norm{\omega}_{L^p(\Omega)} + \norm{d\omega}_{L^p(\Omega)}$. We also let $W^{1,p}_\loc(\wedge^k T^* \Omega)$ be the space of all $\omega \in L^p_\loc(\wedge^k T^* \Omega)$ with a weak exterior derivative $d\omega \in L^p_\loc(\wedge^{k+1} T^* \Omega)$.

Similarly to the usual Sobolev spaces of functions, the spaces $W^{d,p}(\wedge^k T^* \Omega)$ are Banach spaces for $p \in [1, \infty]$. Moreover, by a standard convolution approximation argument, $C^\infty(\wedge^k T^* \Omega)$ is dense in $W^{d,p}(\wedge^k T^* \Omega)$ for $p \in [1, \infty)$. Note that $0$-forms are real-valued functions, and the space $W^{d,p}(\wedge^0 T^* \Omega)$ is in fact precisely the usual Sobolev space $W^{1,p}(\Omega)$. However, for $k \ge 1$, the coefficients $\omega_I$ of a form $\omega \in W^{d,p}(\wedge^k T^* \Omega)$ need not be $W^{1,p}$-functions; a standard example of this is that if $f, g \in L^p(I)$ where $I \subset \R$ is a bounded open interval, then the $1$-form $\omega = f(x) dx + g(y) dy$ is in $W^{d,p}(\wedge^1 T^* I^2)$ with $d \omega = 0$, even though the functions $(x,y) \mapsto f(x)$ and $(x,y) \mapsto g(y)$ are not necessarily Sobolev-regular.

\subsection{Integration of $L^p$-forms}\label{subsect:int_of_Lp_forms}

Suppose that $\Omega \subset \R^n$ is a convex domain. Then every $(k+1)$-tuple of points $(x_0, x_1, \dots, x_k)$, $x_i \in \Omega$, defines a unique oriented $k$-simplex in $\Omega$ with corners at $x_i$. We denote this simplex by $\Delta(x_0, x_1, \dots, x_k)$. We also denote the family of all $k$-simplices in $\Omega$ by $S^k(\Omega)$.

Given $(x_0, x_1, \dots, x_k) \in \Omega^{k+1}$, we have a standard affine map $\phi_{x_0, \dots, x_k}$ from $\Delta_k$ to $\Delta(x_0, x_1, \dots, x_k)$, where $\Delta_k$ is the reference $k$-simplex 
\[
\Delta_k = \{(s_1, \dots, s_k) \in \R^k : s_i \geq 0, s_1 + \dots + s_k \leq 1\}.
\]
The map $\phi_{x_0, \dots, x_k}$ is given by
\[
\phi_{x_0, \dots, x_k}(s_1, \dots, s_k) = x_0 + s_1(x_1 - x_0) + s_2(x_2 - x_0) + \dots + s_k(x_k - x_0).
\]
We note that for any non-negative Borel function $\rho$ on $\Delta(x_0, \dots, x_k)$, we have
\begin{equation}\label{eq:simplex_change_of_vars}
	\int_{\Delta(x_0, \dots, x_k)} \rho \, d\cH^{k} = \int_{\Delta_k} (\rho \circ \phi_{x_0, \dots, x_k}) \smallabs{J_{\phi_{x_0, \dots, x_k}}} \, dm_k,
\end{equation}
where the co-dimensional Jacobian $\smallabs{J_{\phi_{x_0, \dots, x_k}}}$ for a given $(x_0, x_1, \dots, x_k) \in \Omega^{k+1}$ is the constant function
\begin{equation}\label{eq:simplex_Jacobian}
	\smallabs{J_{\phi_{x_0, \dots, x_k}}} = \sqrt{ \det
		\begin{bmatrix} 
			\ip{x_1 - x_0}{x_1 - x_0}  & \dots & \ip{x_1 - x_0}{x_k - x_0}\\
			\vdots & \ddots & \vdots\\
			\ip{x_k - x_0}{x_1 - x_0}  & \dots & \ip{x_k - x_0}{x_k - x_0}
		\end{bmatrix} 
	}.
\end{equation}

We then recall the $p$-modulus of subfamilies of $S^k(\Omega)$; for a more detailed exposition on the $p$-modulus of surface families, we refer the reader to e.g.\ the work of Fuglede \cite{Fuglede_surface-modulus}. If $\cS \subset S^k(\Omega)$ is a family of $k$-simplices, we say that a non-negative Borel function $\rho : \Omega \to [0, \infty]$ is \emph{admissible for $\cS$} if
\[
	\int_{\Delta} \rho d\cH^{k} \geq 1 \text{ for every } \Delta \in \cS.
\]
We denote the set of admissible functions for $\cS$ by $\adm(\cS)$. Then, given $p \in [1, \infty)$, the \emph{$p$-modulus} of $\cS$ is defined by
\[
	\moddens_p(\cS) = \inf \left\{ \int_{\Omega} \rho^p \vol_n : \rho \in \adm(\cS) \right\}.
\]
The $p$-modulus $\moddens_p$ defines an outer measure on $S^k(\Omega)$. We say that $\cS \subset S^k(\Omega)$ is \emph{$p$-exceptional} if $\moddens_p(\cS) = 0$. 

Now, suppose that $\omega \in L^p(\wedge^k T^* \Omega)$ is Borel, with $p \in [1, \infty)$. We define the integral of $\omega$ over a simplex $\Delta(x_0, \dots, x_k) \in S^k(\Omega)$ by
\begin{equation}\label{eq:int_over_simplex}
	\int_{\Delta(x_0, \dots, x_k)} \Omega = \int_{\Delta_k} \phi_{x_0, \dots x_k}^* \omega,
\end{equation}
where $\phi_{x_0, \dots x_k}^* \omega$ is the usual pull-back of the form $\omega$ given by
\[
	\phi_{x_0, \dots x_k}^* \omega
	= \sum\nolimits_I (\omega_I \circ \phi_{x_0, \dots x_k}) d\phi_{x_0, \dots x_k}^{I_1} \wedge \dots \wedge d\phi_{x_0, \dots x_k}^{I_k}.
\] 
Then the integral in \eqref{eq:int_over_simplex} is well-defined and finite outside a $p$-exceptional family of $k$-simplices in $\Omega$. Moreover, changing $\omega$ in a null-set only changes the value of the integral on a $p$-exceptional family of $k$-simplices. For details, we refer e.g.\ to \cite[Section 4.1]{Kangasniemi-Prywes_SurfaceModulus}, or the corresponding results for functions instead of $k$-forms given in \cite[Theorem 3]{Fuglede_surface-modulus}. Because every $\omega \in L^p(\wedge^k T^* \Omega)$ has a Borel representative, we may use this method to define integrals over simplices for $L^p$-forms.

\subsection{Other results}

We also record the following simple fact about convex domains that we use later in the article.

\begin{lemma}\label{lem:convex_approx}
	Suppose that $\Omega \subset \R^n$ is a bounded, convex domain. Then for every $\eps > 0$, the domain $\Omega_\eps = \{x \in \Omega : \dist (x, \partial \Omega) > \eps\}$ is a convex domain that is compactly contained in $\Omega$.
\end{lemma}
\begin{proof}
	The parts that $\Omega_\eps$ is open and compactly contained in $\Omega$ are clear. Hence, the main property to verify is that $\Omega_\eps$ is convex, which will also automatically show that $\Omega_\eps$ is connected. Thus, suppose that $x, y \in \Omega_\eps$, and suppose towards contradiction that there exists $z \in [x, y]$ such that $z \notin \Omega_\eps$. Then, since $d(x, \partial\Omega) > \eps$, $d(y, \partial \Omega) > \eps$, and $d(z, \partial \Omega) \le \eps$ there exists a $w \in \partial \Omega$ such that $d(z, w) < \min(d(x, \partial\Omega), d(y, \partial \Omega))$. By the selection criterion of $w$, we have $x+w-z \in \Omega$ and $y + w - z \in \Omega$. But now $w \notin \Omega$ is on the line segment $[x+w-z, y+w-z] \subset \Omega$, which is a contradiction.
\end{proof}
	
\section{The simplicial integration function of a $k$-form}\label{sect:integration_function}

Let $\Omega \subset \R^n$ be open, and let $\omega \in L^p(\wedge^k T^* \Omega)$ for $p \in [1, \infty)$. By changing $\omega$ in a set of measure zero, we may assume that $\omega$ is Borel. We extend $\omega$ to an element $\omega \in L^p(\wedge^k T^* \R^n)$ by setting $\omega = 0$ outside $\Omega$. As stated in the introduction, we then define a map $I_\omega \colon \Omega^{k+1} \to \R$ by
\begin{equation}\label{eq:I_omega_def}
	I_\omega(x_0, x_1, \dots, x_k) = \int_{\Delta(x_0, x_1, \dots, x_k)} \omega.
\end{equation}
Note that for $0$-forms, i.e.\ functions $f \in L^p(\Omega)$, the map $I_f$ is just the evaluation map $I_f(x_0) = f(x_0)$.

By the discussion in Section \ref{subsect:int_of_Lp_forms}, we have that the integral on the right hand side of \eqref{eq:I_omega_def} is finite outside a $p$-exceptional family of simplices, and that changing $\omega$ in a set of measure zero only changes the integral on a $p$-exceptional family of simplices. Thus, in order for $I_\omega$ to be well-defined, we show that for every $p$-exceptional family of simplices, the corresponding family of corners is a Lebesgue null-set of $\Omega^{k+1}$. This fact is somewhat unsurprising, but nonetheless irritating to prove.

\begin{lemma}\label{lem:modulus_to_measure}
	Let $\Omega \subset \R^n$ be a convex domain, let $P \subset \Omega^{k+1}$, and let $\cS = \{ \Delta(x_0, \dots, x_k) : (x_0, \dots, x_k) \in P\}$. If $\cS$ is $p$-exceptional for some $p \in [1, \infty)$, then $P$ is a Lebesgue null-set.
\end{lemma}
\begin{proof}
	By Fuglede's characterization of $p$-exceptional families of measures, see \cite[Theorem 2]{Fuglede_surface-modulus}, there exists a Borel function $\rho \colon \Omega \to [0, \infty]$ such that $\norm{\rho}_{L^p(\Omega)} < \infty$, but
	\[
		\int_{\Delta} \rho \, d\cH^k = \infty
	\]
	for every $\Delta \in \cS$. By \eqref{eq:simplex_change_of_vars}, this implies that
	\begin{equation}\label{eq:integral_is_infinite}
		\int_{\Delta_k} (\rho \circ \phi_{x_0, \dots, x_k}) \smallabs{J_{\phi_{x_0, \dots, x_k}}} \, dm_k = \infty
	\end{equation}
	for every $(x_0, \dots x_k) \in P$. 
	
	Consider the function
	\[
		\Phi_\omega \colon \Omega^{k+1} \times \Delta_k \to [0, \infty], \,\, \Phi_\omega((x_0, \dots, x_k), x) = \smallabs{J_{\phi_{x_0, \dots, x_k}}} \rho(\phi_{x_0, \dots, x_k}(x)).
	\]
	Since $\rho$ is Borel, the map $((x_0, \dots x_k), x) \mapsto \phi_{x_0, \dots, x_k}(x)$ is continuous, and the map $(x_0, \dots x_k) \mapsto \smallabs{J_{\phi_{x_0, \dots, x_k}}}$ is continuous, we obtain that $\Phi_\omega$ is Borel. Thus, by Fubini's theorem, the function
	\begin{equation}\label{eq:slice_map_1}
		(x_0, \dots, x_k) \mapsto \int_{\Delta_k} \Phi_\omega((x_0, \dots, x_k), x) \, dm_k
	\end{equation}
	is Borel. Now, we let $\tilde{P}$ be the set of all $(x_0, \dots x_k) \in \Omega^{k+1}$ such that \eqref{eq:integral_is_infinite} applies, and note that $\tilde{P}$ is Borel, since it is the pre-image of $\{\infty\}$ under the Borel map given in \eqref{eq:slice_map_1}.
	
	We then prove that $\tilde{P}$ is a null-set, as since $P \subset \tilde{P}$, the claim then follows. Suppose towards contradiction that $m_{(k+1)n}(\tilde{P}) > 0$. Let
	\[
		T \colon \Omega \times (\R^n)^k  \to (\R^n)^{k+1}, \,\, 
		F(x_0, x_1, \dots, x_k) = (x_0, x_0 + x_1, \dots, x_0 + x_k),
	\]
	and note that $T$ is continuous, $T$ is measure-preserving, and $\Omega^{k+1}$ is contained in the image of $T$. Thus, $T^{-1} \tilde{P}$ is a Borel set with positive measure. For $(x_1, \dots, x_k) \in(\R^n)^k$, we let
	\[
		E_{x_1, \dots, x_k} = (T^{-1} \tilde{P}) \cap (\Omega \times \{(x_1, \dots, x_k)\}).
	\]
	Then by Fubini's theorem, $E_{x_1, \dots, x_k}$ is Borel for $m_{kn}$-a.e.\ $(x_1, \dots, x_k) \in (\R^n)^k$, and there exists a Borel set $F \subset (\R^n)^k$ with $m_{kn}(F) > 0$ such that $E_{x_1, \dots, x_k}$ is Borel and $m_n(E_{x_1, \dots, x_k}) > 0$ for every $(x_1, \dots, x_k) \in F$. 
	
	We fix an element $(x_1, \dots, x_k) \in F$. Since $\tilde{P}$ contains a simplex of the form $\Delta(x_0, x_0 + x_1, \dots, x_0 + x_k)$, we must have that $\Delta(0, x_1, \dots, x_k)$ has positive $k$-dimensional measure. Thus, there exists a unique $(n-k)$-dimensional vector space perpendicular to $\Delta(0, x_1, \dots, x_k)$, which we denote $V$. Since $m_n(E_{x_1, \dots, x_k}) > 0$, our third use of Fubini's theorem so far yields a Borel set $W \subset V^{\perp}$ with $\cH^k(W) > 0$ such that $(x + V) \cap E_{x_1, \dots, x_k}$ is Borel and $\cH^{n-k}((x + V) \cap E_{x_1, \dots, x_k}) > 0$ for every $x \in W$.
	
	Now, we fix any $x_0 \in W$, and let 
	\[
		C = \Delta(0, x_1, \dots, x_k) + ((x_0 + V) \cap E_{x_1, \dots, x_k}),
	\]
	noting that $C$ is a measurable set. By definition of $E_{x_1, \dots, x_k}$, we observe that $(x, x+x_1, \dots, x+x_k) \in \tilde{P}$ for every $x \in E_{x_1, \dots, x_k}$. Thus, since $\tilde{P} \subset \Omega^{k+1}$ and $\Omega$ is convex, we have $C \subset \Omega$. Moreover, for $\cH^{n-k}$-a.e.\ $x \in (x_0 + V) \cap E_{x_1, \dots, x_k}$, we have by H\"older's inequality that
	\[
		\int_{x + \Delta(0, x_1, \dots, x_k)} \rho^p \, d\cH^k 
		\geq \frac{1}{[\cH^k(\Delta)]^{p-1}} \left(\int_{\Delta} \rho \, d\cH^k \right)^p 
		= \infty. 
	\]
	Now, since $(x_0 + V) \cap E_{x_1, \dots, x_k}$ is perpendicular to $\Delta(0, x_1, \dots, x_k)$, and since $\cH^{n-k}((x + V) \cap E_{x_1, \dots, x_k}) > 0$, the fourth and final use of Fubini's theorem in this proof yields
	\[
		\int_C \rho^p \, dm_n
		= \int_{(x_0 + V) \cap E_{x_1, \dots, x_k}} \infty \, d \cH^{n-k}
		= \infty.
	\]
	This is a contradiction with $\norm{\rho}_{L^p(\Omega)} < \infty$. Hence, $m_{(k+1)n}(\tilde{P}) > 0$ is false, and the proof is thus complete.
\end{proof}

Thus, together with the discussion of Section \ref{subsect:int_of_Lp_forms}, Lemma \ref{lem:modulus_to_measure} with $\Omega = \R^n$ immediately yields the following corollary.

\begin{cor}\label{cor:I_omega_well_def}
	Let $\Omega \subset \R^n$ be an open set, let $p \in [1, \infty)$, and let $\omega \in L^p(\wedge^k T^*\Omega)$ be Borel, where we extend $\omega$ to all of $\R^n$ by zero extension. Then $I_\omega$ is well-defined and finite for a.e.\ $(x_0, \dots, x_k) \in \Omega^{k+1}$. Moreover, changing $\omega$ in a set of measure zero only changes $I_\omega$ in a set of measure zero.
\end{cor} 

We also show that $I_\omega$ is measurable.

\begin{lemma}\label{lem:I_omega_meas}
	Let $\Omega \subset \R^n$ be an open set, let $p \in [1, \infty)$, and let $\omega \in L^p(\wedge^k T^*\Omega)$ be Borel, where we extend $\omega$ to all of $\R^n$ by zero extension. Then $I_\omega \colon \Omega^{k+1} \to \R$ is Borel.
\end{lemma}
\begin{proof}
	Consider the function $\rho \colon (\R^n)^{k+1} \times \Delta_k \to \R$ defined by
	\[
		\rho(x_0, \dots, x_k, y) \vol_k = (\phi_{x_0, \dots, x_k}^* \omega)_y,
	\]
	where $\phi_{x_0, \dots, x_k} \colon \Delta_k \to \Delta(x_0, \dots, x_k)$ is as in Section \ref{subsect:int_of_Lp_forms}. Then $\rho$ is Borel since $\omega$ is Borel and $(x_0, \dots, x_k, y) \mapsto \phi_{x_0, \dots, x_k}(y)$ is smooth. Moreover,
	\[
		I_\omega(x_0, \dots, x_k) = \int_{\Delta_k} \phi_{x_0, \dots, x_k}^* \omega
		= \int_{\Delta_k} \rho(x_0, \dots, x_k, y) \, dm_k(y).
	\]
	By Corollary \ref{cor:I_omega_well_def}, $y \mapsto \rho(x_0, \dots, x_k, y)$ is integrable for a.e.\ $(x_0, \dots, x_k) \in \Omega^{k+1}$, where we may assume that the exceptional set is Borel by enlarging it. Thus, the fact that $I_\omega$ is Borel measurable follows from a technical fact that is typically wrapped into the statement of Fubini's theorem; for a precise statement of the result we use, see e.g.\ \cite[Corollary 3.4.6]{Bogachev_MeasureTheory}.
\end{proof}

In conclusion, by Corollary \ref{cor:I_omega_well_def} and Lemma \ref{lem:I_omega_meas}, we now have for every open $\Omega \subset \R^n$ a linear map $L^p(\wedge^k T^* \Omega) \to M^k(\Omega, \R^n)$ defined by $\omega \mapsto I_\omega$, where $\omega$ is a chosen Borel representative of its $L^p$-class that has been extended outside $\Omega$ as zero. Moreover, the choice of Borel representative $\omega$ only affects the values of $I_\omega$ in a null-set of $\Omega^{k+1}$.

\section{$L^p$-norms of integration functions}

Let $\Omega \subset \R^n$ be open. As stated in the introduction, we let $M^k(\Omega; \R)$ denote the vector space of measurable functions $F \colon \Omega^{k+1} \to \R$. The elements of $M^k(\Omega; \R)$ are called \emph{measurable (real-valued) $k$-multifunctions} on $\Omega$. In particular, a measurable function $f \colon \Omega \to \R$ is a measurable $0$-multifunction on $\Omega$.

We then fix additional notation related to related to the four seminorms on $M^k(\Omega; \R)$ that were defined in \eqref{eq:Lp_for_multifunct_def}-\eqref{eq:Lp_for_multifunct_def_alt_loc}. As stated in the introduction, given $E \subset \Omega$, $R > 0$, and $c > 0$, we denote
\[
	\B(E, k, R) = \{(x_0, \dots, x_k) \in E^{k+1} : \forall i \in \{1, \dots, k\}, 
	x_i \in \B^n(x_0, R)\}.
\]
and
\[
	E(k, c) := \left\{ (x_0, \dots, x_k) \in E^{k+1} : \max_{i = 1, \dots, k} \abs{x_0 - x_k} < c \dist(x_0, \partial E) \right\}.
\]
Then, for every one of the seminorms \eqref{eq:Lp_for_multifunct_def}-\eqref{eq:Lp_for_multifunct_def_alt_loc}, we define a version for fixed $\theta \in (0, 1)$. In particular, if $F \in M^k(\Omega; \R)$, we denote
\begin{align*}
	\norm{F}_{L^p M^k(E, \theta)}^p 
	&= \int_{E^{k+1}} \frac{(1 - \theta)^k \abs{F(x_0, \dots, x_k)}^p \, dx_0 \ldots dx_k}{\left(\abs{x_1 - x_0} \abs{x_2 - x_0} \cdots \abs{x_k - x_0}\right)^{n + \theta p}},\\
	\norm{F}_{L^p M^k(E\mid R, \theta)}^p 
	&= \int_{\B(E, k, R)} \frac{(1 - \theta)^k \abs{F(x_0, \dots, x_k)}^p \, dx_0 \ldots dx_k}{\left(\abs{x_1 - x_0} \abs{x_2 - x_0} \cdots \abs{x_k - x_0}\right)^{n + \theta p}},\\
	\norm{F}_{L^p M^{k, c}(E, \theta)}^p 
	&= \int_{E(k, c)} \frac{(1 - \theta)^k \abs{F(x_0, \dots, x_k)}^p \, dx_0 \ldots dx_k}{\left(\abs{x_1 - x_0} \abs{x_2 - x_0} \cdots \abs{x_k - x_0}\right)^{n + \theta p}},\\
	\norm{F}_{L^p M^{k, c}(E\mid R, \theta)}^p 
	&= \int_{\B(E, k, R) \cap E(k, c)} \frac{(1 - \theta)^k \abs{F(x_0, \dots, x_k)}^p \, dx_0 \ldots dx_k}{\left(\abs{x_1 - x_0} \abs{x_2 - x_0} \cdots \abs{x_k - x_0}\right)^{n + \theta p}}.
\end{align*}

Since $t \mapsto t^p$ is continuous in $[0, \infty]$, it follows that
\begin{align*}
	\norm{F}_{L^p M^k(E)} &= \limsup_{\theta \to 1^{-}} \norm{F}_{L^p M^k(E, \theta)}, \text{ and}\\
	[F]_{L^p M^k(E)} &= \liminf_{\theta \to 1^{-}} \norm{F}_{L^p M^k(E, \theta)}.
\end{align*}
It is clear that every $\norm{\cdot}_{L^p M^k(E, \theta)}^p$ is a $[0, \infty]$-valued seminorm on $M^k(\Omega; \R)$, as they are just weighted $L^p$-seminorms. Thus, since $\limsup$ is subadditive and commutes with multiplication by non-negative constants, $\norm{\cdot}_{L^p M^k(E)}$ is also a $[0, \infty]$-valued seminorm on $M^k(\Omega; \R)$. We also recall that we say that $\norm{F}_{L^p M^k(E)}$ is obtained as a limit if 
\[
	[F]_{L^p M^k(E)} = \norm{F}_{L^p M^k(E)} = \lim_{\theta \to 1^{-}} \norm{F}_{L^p M^k(E, \theta)}.
\]
Moreover, mutatis mutandis, the aforementioned facts for $\norm{F}_{L^p M^k(E)}$ also apply to all of the other seminorms defined in \eqref{eq:Lp_for_multifunct_def}-\eqref{eq:Lp_for_multifunct_def_alt_loc}, as well as their respective $\liminf$-counterparts.

Clearly, we have for all $\theta \in (0, 1)$ that
\begin{gather*}
	\norm{\cdot}_{L^p M^{k, c}(E\mid R, \theta)} 
		\le \norm{\cdot}_{L^p M^{k}(E\mid R, \theta)}
		\le \norm{\cdot}_{L^p M^{k}(E, \theta)}, \text{ and}\\
	\norm{\cdot}_{L^p M^{k, c}(E\mid R, \theta)} 
		\le \norm{\cdot}_{L^p M^{k, c}(E, \theta)}
		\le \norm{\cdot}_{L^p M^{k}(E, \theta)}.
\end{gather*}
Thus, the above inequalities also pass to the respective $\limsup$ and $\liminf$ -versions.
We also point out that if $E$ is bounded, then for all $\theta \in (0, 1)$, we have
\begin{align}
	\norm{F}_{L^p M^k(E, \theta)} &= \norm{F}_{L^p M^k(E\mid (\diam E), \theta)}, \quad \text{and} \label{eq:bdd_set_equality}\\
	\norm{F}_{L^p M^{k, c}(E, \theta)} &= \norm{F}_{L^p M^{k, c}(E\mid (\diam E), \theta)}.
	\label{eq:bdd_set_equality_alt}
\end{align}

\begin{rem}\label{rem:diagonal behavior}
	Let $\Omega \subset \R^n$ be open, let $R > 0$, and let $F \in M^1(\Omega; \R)$ with $\norm{F}_{L^p M^1(\Omega\mid R)} < \infty$. We claim that if $F$ vanishes in a uniform neighborhood $\B(\Omega, 1, r)$ of the diagonal $\Delta_1(\Omega) \subset \Omega^2$, where $r > 0$, then $\norm{F}_{L^p M^1(\Omega\mid R)} = 0$. Note that if $\Omega$ is bounded, this implies a similar result for $\norm{F}_{L^p M^1(\Omega)}$ by setting $R = \diam \Omega$.
	
	Indeed, since $\norm{F}_{L^p M^1(\Omega\mid R)} < \infty$, there exists a $\theta_0 \in (0, 1)$ for which $\norm{F}_{L^p M^1(\Omega\mid R, \theta_0)} < \infty$. It follows that
	\[
		\int_{\B(\Omega, 1, R)} \abs{F(x,y)}^p \, dx dy \le \frac{R^{n+p\theta_0}}{1-\theta_0} \norm{F}_{L^p M^1(\Omega\mid R, \theta_0)}^p < \infty.
	\]
	Now, if $(x,y) \notin \B(\omega, 1, r)$, then $\abs{x-y} \ge r$. Thus, using the fact that $F$ vanishes in $\B(\omega, 1, r)$, we obtain that
	\begin{multline*}
		\norm{F}_{L^p M^1(\Omega\mid R)}^p 
		\le \limsup_{\theta \to 1^{-}} \frac{1-\theta}{r^{np+\theta}}
			\int_{\B(\Omega, 1, R)} \abs{F(x,y)}^p \, dx dy\\
		\le \limsup_{\theta \to 1^{-}} \frac{1-\theta}{\min(1, r^{np+1})} \frac{R^{n+p\theta_0}}{1-\theta_0} \norm{F}_{L^p M^1(\Omega\mid R, \theta_0)}^p
		= 0,
	\end{multline*}
	completing the proof of our claim.
	
	As stated in the introduction, a similar argument for $F \in M^k(\Omega; \R)$ with $k > 1$ becomes more complicated. The reason is that in the complement of $\B(\Omega, k, r)$, one only has an estimate of the form $\abs{x_0 - x_i} \ge r$ for one coordinate $i$ at a time. After such an estimate, one does have an extra $(1-\theta)$-term compared to the amount of $\abs{x_0 - x_i}$-terms, but one would have to rule out an unexpectedly fast-growing singularity caused by the other $\abs{x_0 - x_i}$-terms, possibly using an extra assumption on $F$. For integration functions $I_\omega$ of $L^p$-integrable $k$-forms, such an argument is possible; this is shown in Lemma \ref{lem:only_close_pts_matter}.
\end{rem}

\subsection{The first two cases of Theorem \ref{thm:equiv_of_norms_general}}

Our main objective in this section is to prove Theorem \ref{thm:equiv_of_norms_general}. We begin by proving the part of the result for $\norm{\cdot}_{L^p M^k(\Omega\mid R)}$. By \eqref{eq:bdd_set_equality}, the part of the result for $\norm{\cdot}_{L^p M^k(\Omega)}$ will follow suit. 

We start with the following version of the upper bound, which will be used in the main proof to enable approximation arguments and dominated convergence use.

\begin{lemma}\label{lem:dom_conv_enabler}
	Let $\Omega \subset \R^n$ be an open set, let $\omega \in L^p(\wedge^k T^* \Omega)$, and let $R > 0$. Then there exists a constant $C = C(n, p, k)$ such that for every $\theta \in (0, 1)$,
	\begin{equation}\label{eq:diam_dependent_Lp_bound}
		\norm{I_\omega}_{L^p M^k(\Omega\mid R, \theta)} \leq C R^{k(1-\theta)} \norm{\omega}_{L^p(\Omega)}.
	\end{equation}
	In particular,
	\[
		\norm{I_\omega}_{L^p M^k(\Omega\mid R)}
		\leq C \norm{\omega}_{L^p(\Omega)}.
	\]
\end{lemma}
\begin{proof}
	Recall that $I_\omega$ is defined by extending $\omega$ to all of $\R^n$ by setting $\omega = 0$ outside $\Omega$. With H\"older's inequality and the definition of $I_\omega$, we obtain
	\begin{multline*}
		\norm{I_\omega}_{L^p M^k(\Omega\mid R, \theta)}^p\\
		\lesssim_{p, k}
		\int_{\Delta_k \times \B(\Omega, k, R)} \frac{(1-\theta)^{k} \smallabs{\omega_{\phi_{x_0, \dots, x_k}(s)}(x_1 - x_0, \dots, x_k - x_0)}^p}{(\abs{x_1 - x_0} \cdots \abs{x_k - x_0})^{n + p\theta}}
		\, ds dx_0 \dots dx_k.
	\end{multline*}
	We then note that $\abs{\omega_x(v_1, \dots, v_k)} \leq \abs{\omega_x} \abs{v_1} \cdots \abs{v_k}$. Applying this, we have
	\begin{multline*}
		\int_{\Delta_k \times \B(\Omega, k, R)} \frac{(1-\theta)^{k} \smallabs{\omega_{\phi_{x_0, \dots, x_k}(s)}(x_1 - x_0, \dots, x_k - x_0)}^p}{(\abs{x_1 - x_0} \cdots \abs{x_k - x_0})^{n + p\theta}}
		\, ds dx_0 \dots dx_k\\
		\leq \int_{\Delta_k \times \B(\Omega, k, R)} \frac{(1-\theta)^{k} \smallabs{\omega_{\phi_{x_0, \dots, x_k}(s)}}^p}{(\abs{x_1 - x_0} \cdots \abs{x_k - x_0})^{n - p(1 - \theta)}}
		\, ds dx_0 \dots dx_k.
	\end{multline*}
	We then denote $y_i = x_i - x_0$, $i \in \{1, \dots, k\}$. Note that in our domain of integration, we have $y_i \in \B^n(0, R)$ for all $i$. Thus, by a change of variables, we get
	\begin{multline*}
		\int_{\Delta_k \times \B(\Omega, k, R)} \frac{
				(1-\theta)^{k} \smallabs{\omega_{\phi_{x_0, \dots, x_k}(s)}}^p
			}{
				(\abs{x_1 - x_0} \cdots \abs{x_k - x_0})^{n - p(1 - \theta)}
			} \, ds dx_0 \dots dx_k \\
		\leq \int_{\Delta_k \times \Omega \times [\B^n(0, R)]^k}
			\frac{(1-\theta)^k \abs{\omega_{x_0 + s_1 y_1 + \dots + s_k y_k}}^p}{(\abs{y_1} \cdots \abs{y_k})^{n-p(1-\theta)}} 
			\,  ds dx_0 dy_1 \dots dy_k.
	\end{multline*}
	We then use the Fubini-Tonelli theorem to compute the $dx_0$-integral first. Due to the fact that $\omega$ has been zero extended outside $\Omega$, we obtain the estimate
	\begin{multline*}
		\int_{\Delta_k \times \Omega \times [\B^n(0, R)]^k}
		\frac{(1-\theta)^k \abs{\omega_{x_0 + s_1 y_1 + \dots + s_k y_k}}^p}{(\abs{y_1} \cdots \abs{y_k})^{n-p(1-\theta)}} 
		\, ds dx_0 dy_1 \dots dy_k\\
		\le \int_{\Delta_k \times [\B^n(0, R)]^k}
		\frac{(1-\theta)^k \norm{\omega}_{L^p(\Omega)}^p}{(\abs{y_1} \cdots \abs{y_k})^{n-p(1-\theta)}} 
		\, ds dy_1 \dots dy_k
	\end{multline*} 
	Then, by using polar coordinates on the variables $y_i$, we conclude that 
	\begin{multline*}
		\int_{\Delta_k \times [\B^n(0, R)]^k}
		\frac{(1-\theta)^k \norm{\omega}_{L^p(\Omega)}^p}{(\abs{y_1} \cdots \abs{y_k})^{n-p(1-\theta)}} 
		\, ds dy_1 \dots dy_k\\
		= m_k(\Delta_k) \left( \cH^{n-1}(\S^{n-1}) \int_0^R (1-\theta) r^{p(1-\theta) - 1} \, dr \right)^k \norm{\omega}^p_{L^p(\Omega)}\\
		= \frac{m_k(\Delta_k) \left( \cH^{n-1}(\S^{n-1}) \right)^k}{p^k} R^{kp(1-\theta)} \norm{\omega}^p_{L^p(\Omega)},
	\end{multline*}
	completing the proof of \eqref{eq:diam_dependent_Lp_bound}. The claimed bound on $\norm{I_\omega}_{L^p M^k(\Omega\mid R)}$ then follows by letting $\theta \to 1^{-}$ in the right hand side of \eqref{eq:diam_dependent_Lp_bound}.
\end{proof}

Next, we record a simple technical lemma about limits that nevertheless is used sufficiently many times to warrant a standalone statement.

\begin{lemma}\label{lem:limit_switching_lemma}
	Let $\Omega \subset \R^n$ be a bounded domain, let $\omega\in L^p(\wedge^k T^* \Omega)$, and let $R > 0$. Suppose that a sequence of $k$-forms $\omega_j$ converges to $\omega$ in $L^p(\wedge^k T^* \Omega)$. Moreover, suppose that for every $\theta \in (0, 1)$, $\norm{\cdot}_\theta$ is a seminorm on $M^k(\Omega; \R)$ with $\norm{\cdot}_\theta \le \norm{\cdot}_{L^p M^k (\Omega\mid R, \theta)}$. If $\lim_{\theta \to 1^{-}} \norm{I_{\omega_j}}_\theta$ exists for every $\theta \in (0, 1)$, then $\lim_{\theta \to 1^{-}} \norm{I_{\omega}}_\theta$ exists, and moreover
	\[
		\lim_{\theta \to 1^{-}} \norm{I_{\omega}}_\theta 
		= \lim_{j \to \infty} \lim_{\theta \to 1^{-}} \norm{I_{\omega_j}}_\theta.
	\]
\end{lemma}
\begin{proof}
	We compute using Lemma \ref{lem:dom_conv_enabler} and our assumptions on $\norm{\cdot}_\theta$ that
	\begin{multline*}
		\abs{\norm{I_\omega}_{\theta} - \norm{I_{\omega_j}}_{\theta}}
		\leq \norm{I_\omega - I_{\omega_j}}_{\theta}
		= \norm{I_{\omega - \omega_j}}_{\theta}\\
		\le \norm{I_{\omega - \omega_j}}_{L^p M^k(\Omega\mid R, \theta)}
		\leq C(p, n, k) \max(R^k, 1) \norm{\omega - \omega_j}_{L^p(\Omega)}.
	\end{multline*}
	Hence, we have $\lim_{j \to \infty} \norm{I_{\omega_j}}_{\theta} = \norm{I_\omega}_\theta$ uniformly in $\theta$. Thus, the claim follows by using the Moore-Osgood theorem of interchanging limits.
\end{proof}

Following this, we show that for $\omega \in L^p(\wedge^k T^* \Omega)$, the definition of the seminorm $\norm{I_\omega}_{L^p M^k(\Omega)}$ only depends on tuples $(x_0, \dots, x_k)$ where the points are sufficiently close to each other. 

\begin{lemma}\label{lem:only_close_pts_matter}
	Let $\Omega \subset \R^n$ be a bounded domain, let $R > 0$, and let $\omega \in L^p(\wedge^k T^* \Omega)$, where we extend $\omega$ to $L^p(\wedge^k T^* \R^n)$ by setting $\omega = 0$ outside $\Omega$. Then, for
	\begin{equation}\label{eq:eps_theta_def}
		\eps_\theta = e^{-1/\sqrt{1-\theta}},
	\end{equation} 
	we have
	\begin{equation}\label{eq:large_dist_integral}
		\lim_{\theta \to 1^{-}} \left( \norm{I_\omega}_{L^p M^k(\Omega\mid R, \theta)}^p - \norm{I_\omega}_{L^p M^k(\Omega\mid R\eps_\theta, \theta)}^p  \right) = 0.
	\end{equation}
\end{lemma}
\begin{proof}
	We first show that we may assume that $\omega \in C_0(\wedge^k T^* \Omega)$. Indeed, since  $C_0(\wedge^k T^* \Omega)$ is dense in $L^p(\wedge^k T^* \Omega)$, we can take an approximating sequence $\omega_j \in C_0(\wedge^k T^* \Omega)$ of $\omega$. For $F \in M^k(\Omega; \R)$, we denote
	\begin{multline*}
		\norm{F}_\theta^p := \norm{F}_{L^p M^k(\Omega\mid R, \theta)}^p - \norm{F}_{L^p M^k(\Omega\mid R\eps_\theta, \theta)}^p\\
		= \int_{\B(\Omega, k, R) \setminus \B(\Omega, k, R\eps_\theta)} \frac{(1 - \theta)^k \abs{F(x_0, \dots, x_k)}^p}{\left(\abs{x_1 - x_0} \cdots \abs{x_k - x_0}\right)^{n + \theta p}} \, dx_0 \ldots dx_k.
	\end{multline*}
	Then $\norm{\cdot}_\theta$ is a seminorm on $M^k(\Omega; \R)$ with $\norm{\cdot}_\theta \le \norm{\cdot}_{L^p M^k(\Omega\mid R, \theta)}$. Thus, by Lemma \ref{lem:limit_switching_lemma} and the assumption that \eqref{eq:large_dist_integral} holds for $\omega_j$, we have 
	\[
		\lim_{\theta \to 1^{-}} \norm{\omega}_\theta = \lim_{j \to \infty} 0 = 0,
	\]
	implying that \eqref{eq:large_dist_integral} holds for $\omega$. 
	
	Thus, we assume $\omega \in C_0(\wedge^k T^* \Omega)$. We note that by the Fubini-Tonelli theorem, we have
	\begin{multline*}
		\int_{\B(\Omega, k, R) \setminus \B(\Omega, k, R\eps_\theta)} \frac{(1 - \theta)^k \abs{I_\omega(x_0, \dots, x_k)}^p}{\left(\abs{x_1 - x_0} \abs{x_2 - x_0} \cdots \abs{x_k - x_0}\right)^{n + \theta p}} \, dx_0 \dots dx_k\\
		\le \int_{\Omega} \int_{(\B^n(x_0, R))^k \setminus (\B^n(x_0, R\eps_\theta))^k} 
		\frac{(1-\theta)^k \abs{I_\omega(x_0, \dots, x_k)}^p}{(\abs{x_0 - x_1} \cdots \abs{x_0  - x_k})^{n+p\theta}} \, dx_1 \dots dx_k dx_0.
	\end{multline*} 
	We fix $x_0 \in \Omega$, and similarly as in the proof of Lemma \ref{lem:dom_conv_enabler}, we make changes of variables $y_i = x_i - x_0$ for $i \in \{1, \dots, k\}$ to obtain
	{\allowdisplaybreaks\begin{align}\label{eq:change_of_vars_step}
			&\int_{(\B^n(x_0, R))^k \setminus (\B^n(x_0, R\eps_\theta))^k} 
				\frac{(1-\theta)^k \abs{I_\omega(x_0, \dots, x_k)}^p}{(\abs{x_0 - x_1} \cdots \abs{x_0  - x_k})^{n+p\theta}} \, dx_1 \dots dx_k\\
			&\begin{multlined}
				\quad= \int_{(\B^n(0, R))^k \setminus (\B^n(0, D\eps_\theta))^k} 
				\frac{(1-\theta)^k}{(\abs{y_1} \cdots \abs{y_k})^{n+p\theta}} \\
				\hspace{1.4cm}\abs{ \int_{\Delta_k} \omega_{x_0 + s_1 y_1 + \dots + s_k y_k}
					\left( y_1, \dots, y_k \right) ds }^p dy_1 \dots dy_k
			\end{multlined}\nonumber\\
			&\begin{multlined}
				\quad= \int_{(\B^n(0, R))^k \setminus (\B^n(0, D\eps_\theta))^k} 
				\frac{(1-\theta)^k}{(\abs{y_1} \cdots \abs{y_k})^{n-p(1 -\theta)}} \\
				\hspace{1.4cm} \abs{\int_{\Delta_k} \omega_{x_0 + s_1 y_1 + \dots + s_k y_k}
					\left( \frac{y_1}{\abs{y_1}}, \dots, \frac{y_k}{\abs{y_k}} \right) ds }^p
				dy_1 \dots dy_k.
			\end{multlined}\nonumber
	\end{align}}
	We then observe that $\lim_{\theta \to 1^{-}} \eps_\theta^{p(1-\theta)} = 1$. Thus, using the fact that $\omega \in C_0(\wedge^k T^* \Omega)$ implies $\norm{\omega}_{L^\infty(\R^n)} < \infty$, we obtain for every $x_0 \in \Omega$ and $i \in \{1, \dots, k\}$ that
	\begin{align*}
			&\int_{(\B^n(x_0, R))^{k-1}}
			\int_{(\B^n(x_0, R))\setminus \B^n(0,D\eps_\theta)}  \frac{(1-\theta)^k}{(\abs{y_1} \cdots \abs{y_k})^{n-p(1-\theta)}}\\
			&\qquad \abs{ \int_{\Delta_k} \omega_{x_0 + s_1 y_1 + \dots + s_k y_k} \left( \frac{y_1}{\abs{y_1}}, \dots, \frac{y_k}{\abs{y_k}} \right) ds }^p \, dy_i dy_1 \dots \hat{dy_i} \dots dy_k\\
			&\quad\leq C(n, k, p)  \norm{\omega}^p_{L^\infty(\R^n)}
			R^{(k-1)p(1-\theta)} 
			\left( \int_{R\eps_\theta}^R (1-\theta) r^{p(1-\theta) - 1} \, dr \right)\\
			&\quad= C(n, k, p) \norm{\omega}^p_{L^\infty(\R^n)} R^{kp(1-\theta)} \left( 1 - \eps_\theta^{p(1-\theta)}  \right)
			\xrightarrow[\theta \to 1^{-}]{} 0.
	\end{align*}
	
	Thus, we have
	\begin{multline}\label{eq:integral_limit_1}
		\lim_{\theta \to 1^{-}} \int_{(\B^n(x_0, R))^k \setminus (\B^n(x_0, R\eps_\theta))^k} \frac{(1 - \theta)^k \abs{I_\omega(x_0, \dots, x_k)}^p \, dx_1 \dots dx_k}{\left(\abs{x_1 - x_0} \cdots \abs{x_k - x_0}\right)^{n + \theta p}}\\  = 0
	\end{multline}
	for all $x_0 \in \Omega$, and the convergence is uniform with respect to $x_0$. In addition to this, we note that since $\omega \in C_0(\wedge^k T^* \Omega)$, the integral in \eqref{eq:integral_limit_1} vanishes when $x_0$ is outside the compact set $\B^n(\spt \omega, R)$. Thus, the claimed \eqref{eq:large_dist_integral} follows.
\end{proof}

We are ready to prove the first part of Theorem \ref{thm:equiv_of_norms_general}.

\begin{prop}\label{prop:equiv_case_unbdd_regular}
	Let $\Omega \subset \R^n$ be open, let $R > 0$, and let $\omega \in L^p(\wedge^k T^* \Omega)$, where we extend $\omega$ to $L^p(\wedge^k T^* \R^n)$ by setting $\omega = 0$ outside $\Omega$. 
	Then there exists a constant $K = K(p, k)$ such that
	\[
		[I_\omega]_{L^p M^k(\Omega\mid R)} 
		=\norm{I_\omega}_{L^p M^k(\Omega\mid R)} 
		= K \bigl\lVert\abs{\omega}_{\S,p}\bigr\rVert_{L^p(\Omega)},
	\]
	where the norm $\abs{\alpha}_{\S,p}$ for $k$-covectors $\alpha \in \operatorname{Alt}_k(V)$ on an $n$-dimensional normed space $V$ is defined by
	\[
		\abs{\alpha}_{\S,p} = \left( \int_{\{\abs{v_1} = \ldots = \abs{v_k} = 1\}} \abs{\alpha(v_1, \dots, v_k)}^p \, d\cH^{n-1}(v_1) \dots d\cH^{n-1}(v_k) \right)^\frac{1}{p}.
	\]
	In particular,  $\norm{I_\omega}_{L^p M^k(\Omega\mid R)}$ is obtained as a limit, and there exists a constant $C = C(p, n, k)$ such that
	\[
		C^{-1} \norm{\omega}_{L^p(\Omega)} \leq \norm{I_\omega}_{L^p M^k(\Omega\mid R)}
		\leq C \norm{\omega}_{L^p(\Omega)}.
	\]
\end{prop}

\begin{proof}
	We start as in Lemma \ref{lem:only_close_pts_matter}, by showing that we may assume that $\omega \in C_0(\wedge^k)$. Indeed, suppose the result has been shown for such $\omega$. Since $C_0(\wedge^k T^* \Omega)$ is dense in $L^p(\wedge^k T^* \Omega)$, we can take an approximating sequence $\omega_j \in C_0(\wedge^k T^* \Omega)$ of $\omega$. By our assumption that the theorem statement holds for $\omega_j$, we have that $\norm{I_{\omega_j}}_{L^p M^k(\Omega\mid R)}$ is obtained as a limit for every $j$. Thus, we can use Lemma \ref{lem:limit_switching_lemma} with $\norm{\cdot}_\theta = \norm{\cdot}_{L^p M^k(\Omega\mid R, \theta)}$ to obtain that
	\[
		\lim_{\theta \to 1^{-1}} \norm{I_\omega}_{L^p M^k(\Omega\mid R, \theta)}
		= \lim_{j \to \infty} \norm{I_{\omega_j}}_{L^p M^k(\Omega \mid R)}
		= \lim_{j \to \infty} K \lVert\abs{\omega_j}_{\S,p}\rVert_{L^p(\Omega)}.
	\] 
	On the other hand, since $\abs{\cdot}_{\S, p}$ is comparable with the Grassmann norm $\abs{\cdot}$ due to both being norms on a finite-dimensional vector space, we have
	\begin{multline*}
		\big\lvert \smallnorm{\abs{\omega}_{\S,p}}_{L^p(\Omega)} - 
			\smallnorm{\abs{\omega_n}_{\S,p}}_{L^p(\Omega)} \big\rvert
		\leq \big\lVert \abs{\omega}_{\S,p} - \abs{\omega_n}_{\S,p} \big\rVert_{L^p(\Omega)}\\
		\leq \big\lVert \abs{\omega - \omega_n}_{\S,p} \big\rVert_{L^p(\Omega)}
		\leq C(p,n,k) \norm{\omega-\omega_n}_{L^p(\Omega)}
		\xrightarrow[n \to \infty]{} 0.
	\end{multline*} 
	Thus, we also have $\lim_{j \to \infty} \lVert\abs{\omega_j}_{\S,p}\rVert_{L^p(\Omega)} = \lVert\abs{\omega}_{\S,p}\rVert_{L^p(\Omega)}$, completing the reduction to the continuous compactly supported case.
	
	Thus, we may assume $\omega \in C_0(\wedge^k T^* \Omega) \cap L^p(\wedge^k T^* \Omega)$. By Lemma \ref{lem:only_close_pts_matter} and the Fubini-Tonelli theorem, we have
	\begin{multline}\label{eq:U_reduction_step}
		\lim_{\theta \to 1^{-}} \norm{I_\omega}_{L^p M^k(\Omega\mid R, \theta)}^p
		= \lim_{\theta \to 1^{-}} \norm{I_\omega}_{L^p M^k(\Omega\mid R\eps_\theta, \theta)}^p\\
		= \lim_{\theta \to 1^{-}} \int_{\Omega} \int_{\Omega^k \cap (\B^n(x_0, R\eps_\theta))^k} \frac{(1 - \theta)^k \abs{I_\omega(x_0, \dots, x_k)}^p}{\left(\abs{x_1 - x_0} \cdots \abs{x_k - x_0}\right)^{n + \theta p}} \, dx_1\dots dx_k dx_0
	\end{multline}
	if the latter limit exists. 
	We then fix $x_0 \in \Omega$. Noting that $\lim_{\theta \to 1^{-}} \eps_\theta = 0$, we suppose that $\theta$ is close enough to one that $\B^n(0,R\eps_\theta) \subset \Omega$. Similarly as in \eqref{eq:change_of_vars_step}, we get via a change of variables $y_i = x_i - x_0$ that
	{\allowdisplaybreaks\begin{align*}
		&\int_{\Omega^k \cap (\B^n(x_0, R\eps_\theta))^k} \frac{(1 - \theta)^k \abs{I_\omega(x_0, \dots, x_k)}^p}{\left(\abs{x_1 - x_0} \cdots \abs{x_k - x_0}\right)^{n + \theta p}} \, dx_1\dots dx_k\\
		&\begin{multlined}
			\qquad= \int_{(\B^n(0, R\eps_\theta))^k} 
				\frac{(1-\theta)^k}{(\abs{y_1} \cdots \abs{y_k})^{n+p\theta}} \\
			\hspace{3cm} \abs{\int_{\Delta_k} \omega_{x_0 + s_1 y_1 + \dots + s_k y_k}
				\left( \frac{y_1}{\abs{y_1}}, \dots, \frac{y_k}{\abs{y_k}} \right) ds }^p
			dy_1 \dots dy_k.
		\end{multlined}
	\end{align*}}
	
	We can then again apply polar coordinates to $y_i$ to obtain
		\begin{equation*}\begin{aligned}
			&\int_{[\B^n(0,R\eps_\theta)]^k}  \frac{(1-\theta)^k}{(\abs{y_1} \cdots \abs{y_k})^{n-p(1-\theta)}}\\
			&\hspace{1.6cm} \abs{ \int_{\Delta_k} \omega_{x_0 + s_1 y_1 + \dots + s_k} \left( \frac{y_1}{\abs{y_1}}, \dots, \frac{y_k}{\abs{y_k}} \right) ds }^p \, dy_1 \dots dy_k\\
			&\quad= \int_{(\S^{n-1})^k}  \int_{[0, R\eps_\theta]^k} (1-\theta)^k (r_1 \cdots r_k)^{p(1-\theta) - 1}\\
			&\quad\hspace{1.6cm}\abs{ \int_{\Delta_k} \omega_{x_0 + s_1 r_1 v_1 + \dots + s_k r_k v_k} \left( v_1, \dots, v_k \right) ds }^p
			dr_1 \dots dr_k dv_1 \dots dv_k,
	\end{aligned}\end{equation*}
	where the $v_i$-integrals are respect to spherical volume. Now, since $\omega \in C_0(\wedge^k T^*\R^n)$, its coefficients are uniformly continuous, and thus
	\[
		\lim_{\max(\abs{r_i}) \to 0 }\int_{\Delta_k} \omega_{x_0 + s_1 r_1 v_1 + \dots + s_k r_k v_k} \left( v_1, \dots, v_k \right) ds
		= m_k(\Delta_k) \omega_{x_0}(v_1, \dots, v_k)
	\]
	uniformly with respect to $x_0$ and $(v_1, \dots, v_k)$. Because of this and the previously used fact that
	\[
		\lim_{\theta \to 1^{-}} \int_0^{R\eps_\theta} (1-\theta) r^{p(1-\theta) - 1} \, dr
		= \lim_{\theta \to 1^{-}} \frac{1}{p} (R\eps_\theta)^{p(1-\theta)} = \frac{1}{p},
	\] 
	we obtain
	\begin{equation}\label{eq:integral_breakdown_part_2}\begin{aligned}
			&\lim_{\theta \to 1^{-}}\int_{[\B^n(0,R\eps_\theta)]^k}  \frac{(1-\theta)^k}{(\abs{y_1} \cdots \abs{y_k})^{n-p(1-\theta)}}\\
			&\quad\qquad \abs{ \int_{\Delta_k} \omega_{x_0 + s_1 y_1 + \dots + s_k} \left( \frac{y_1}{\abs{y_1}}, \dots, \frac{y_k}{\abs{y_k}} \right) ds }^p \, dy_1 \dots dy_k\\
			&\quad= \int_{(\S^{n-1})^k} p^{-k} m_k^p(\Delta_k) \abs{\omega_{x_0}(v_1, \dots, v_k)}^p \, dv_1 \dots dv_k\\
			&\quad= p^{-k} m_k^p(\Delta_k) \abs{\omega_{x_0}}_{\S, p}^p,
	\end{aligned}\end{equation}
	where the convergence is uniform with respect to $x_0$.
	
	In conclusion, we have
	\begin{multline}\label{eq:integral_limit_2}
		\lim_{\theta \to 1^{-}} \int_{\Omega^k \cap (\B^n(x_0, R\eps_\theta))^k} \frac{(1 - \theta)^k \abs{I_\omega(x_0, \dots, x_k)}^p}{\left(\abs{x_1 - x_0}  \cdots \abs{x_k - x_0}\right)^{n + \theta p}} \, dx_1 \dots dx_k\\
		=  p^{-k} m_k^p(\Delta_k) \abs{\omega_{x_0}}_{\S, p}^p
	\end{multline}
	for all $x_0 \in \Omega$, and this convergence is uniform in $x_0$. Moreover, similarly to the end of the proof of Lemma \ref{lem:only_close_pts_matter}, both $\omega_{x_0}$ and the integrals on the left hand side of \eqref{eq:integral_limit_2} vanish when $x_0$ is outside the compact set $\B^n(\spt \omega, R)$. Thus, we obtain that
	\begin{multline*}
		\lim_{\theta \to 1^{-}} \int_{\Omega} \int_{\Omega^k \cap (\B^n(x_0, R\eps_\theta))^k} \frac{(1 - \theta)^k \abs{I_\omega(x_0, \dots, x_k)}^p}{\left(\abs{x_1 - x_0} \cdots \abs{x_k - x_0}\right)^{n + \theta p}} \, dx_1\dots dx_k dx_0\\
		= p^{-k} m_k^p(\Delta_k) \int_{\Omega}  \abs{\omega}_{\S, p}^p \, \vol_n.
	\end{multline*}
	By combining this with \eqref{eq:U_reduction_step}, we conclude that 
	\[
		\norm{I_\omega}_{L^p M^k(\Omega)} = p^{-\frac{k}{p}} m_k(\Delta_k) \smallnorm{ \abs{\omega}_{\S, p} }_{L^p(\Omega)}
	\]
	and that $\norm{I_\omega}_{L^p M^k(\Omega)}$ is obtained as a limit. Since $\abs{\omega}_{\S, p}$ and $\abs{\omega}$ are uniformly comparable, the claimed two-sided estimate also follows immediately.
\end{proof}

We immediately obtain the following case of Theorem \ref{thm:equiv_of_norms_general} as a corollary, which notably completes the proof of Theorem \ref{thm:equiv_of_norms_basic}.

\begin{cor}\label{cor:equiv_case_bdd_regular}
	Let $\Omega \subset \R^n$ be open and bounded, and let $\omega \in L^p(\wedge^k T^* \Omega)$, where we extend $\omega$ to $L^p(\wedge^k T^* \R^n)$ by setting $\omega = 0$ outside $\Omega$. Then
	\[
		[I_\omega]_{L^p M^k(\Omega)} 
		=\norm{I_\omega}_{L^p M^k(\Omega)} 
		= K \bigl\lVert\abs{\omega}_{\S,p}\bigr\rVert_{L^p(\Omega)},
	\]
	where $K = K(p, k)$ and $\abs{\cdot}_{\S,p}$ are as in Proposition \ref{prop:equiv_case_unbdd_regular}.
\end{cor}
\begin{proof}
	Due to \eqref{eq:bdd_set_equality}, the claim follows from Proposition \ref{prop:equiv_case_unbdd_regular} with $R = \diam \Omega$.
\end{proof}

\subsection{The last two cases of Theorem \ref{thm:equiv_of_norms_general}}

Two cases of Theorem \ref{thm:equiv_of_norms_general} remain. Our strategy is again to resolve the case $\norm{\cdot}_{L^p M^{k, c}(\Omega\mid R)}$, and then apply \eqref{eq:bdd_set_equality_alt} to obtain the result for $\norm{\cdot}_{L^p M^{k, c}(\Omega)}$. Before we start proving the cases, we require the following lemma.

\begin{lemma}\label{lem:compact_support_lemma}
	Let $\Omega \subset \R^n$ be open, let $R > 0$, and let $\omega \in L^p(\wedge^k T^* \Omega)$, where we extend $\omega$ to $L^p(\wedge^k T^* \R^n)$ by setting $\omega = 0$ outside $\Omega$. Suppose that $\spt \omega$ is a compact subset of $\Omega$. Then for every open $V \subset \Omega$ with $\spt \omega \subset V$, we have
	\[
		[I_\omega]_{L^p M^k(V\mid R)} = \norm{I_\omega}_{L^p M^k(V\mid R)} = \norm{I_\omega}_{L^p M^k(\Omega \mid R)}.
	\]
\end{lemma}
\begin{proof}
	Note that the claim is trivial if $V = \R^n$, since then also $\Omega = \R^n$. Let $r = \dist(\spt \omega, \R^n \setminus V)$. Since $V$ is open, we have $r > 0$. It is clear that 
	\[
		\norm{I_\omega}_{L^p M^k(V\mid R)} 
		\le \norm{I_\omega}_{L^p M^k(\Omega\mid R)},
	\]
	so in order to prove the claim, we need to show that
	\[
		[I_\omega]_{L^p M^k(V\mid R)}
		\ge \norm{I_\omega}_{L^p M^k(\Omega\mid R)}.
	\]
	Since $\norm{I_\omega}_{L^p M^k(\Omega\mid R)}$ is obtained as a limit due to Proposition \ref{prop:equiv_case_unbdd_regular}, we obtain the estimate
	\begin{align*}
		&\norm{I_\omega}_{L^p M^k(\Omega\mid R)}^p - [I_\omega]_{L^p M^k(V\mid R)}^p\\
		&\quad= \liminf_{\theta \to 1^{-}} \norm{I_\omega}_{L^p M^k(\Omega\mid R, \theta)}^p - \liminf_{\theta \to 1^{-}} \norm{I_\omega}_{L^p M^k(V\mid R, \theta)}^p\\
		&\quad\le \limsup_{\theta \to 1^{-}} \int_{\B(\Omega, k, R) \setminus \B(V, k, R)} \frac{(1 - \theta)^k \abs{I_\omega(x_0, \dots, x_k)}^p}{\left(\abs{x_1 - x_0} \cdots \abs{x_k - x_0}\right)^{n + \theta p}} \, dx_0 \dots dx_k.
	\end{align*}
	We then apply Lemma \ref{lem:only_close_pts_matter} to conclude that
	\begin{multline}\label{eq:reduction_to_U-set}
		\limsup_{\theta \to 1^{-}} \int_{\B(\Omega, k, R) \setminus \B(V, k, R)} \frac{(1 - \theta)^k \abs{I_\omega(x_0, \dots, x_k)}^p}{\left(\abs{x_1 - x_0} \cdots \abs{x_k - x_0}\right)^{n + \theta p}} \, dx_0 \dots dx_k\\
		= \limsup_{\theta \to 1^{-}} \int_{\B(\Omega, k, R\eps_\theta) \setminus \B(V, k, R)} \frac{(1 - \theta)^k \abs{I_\omega(x_0, \dots, x_k)}^p}{\left(\abs{x_1 - x_0} \cdots \abs{x_k - x_0}\right)^{n + \theta p}} \, dx_0 \dots dx_k.
	\end{multline}
	
	Now, noting that $\lim_{\theta \to 1^{-}} \eps_\theta = 0$, suppose that $\theta$ is close enough to $1$ that $R\eps_\theta < r/2$, and let $(x_0, \dots, x_k) \in \B(\Omega, k, R\eps_\theta) \setminus \B(V, k, R)$. Since $\B(\Omega, k, R\eps_\theta) \cap  V^{k+1} = \B(V, k, R\eps_\theta) \subset \B(V, k, R)$, we must have $(x_0, \dots, x_k) \notin V^{k+1}$. Hence, there exists an index $i$ such that $\dist(x_i, \spt \omega) \ge r$.	Moreover, since $(x_0, \dots, x_k) \in \B(\Omega, k, R\eps_\theta)$, we have $\abs{x_i - x_0} < R\eps_\theta < r/2$, from which it follows that $\dist(x_0, \spt \omega) > r/2 > R\eps_\theta$. 
	
	Thus, $\omega$ is identically zero in the ball $\B^n(x_0, R\eps_\theta)$ which contains all of the points $x_j, j \in \{0, \dots, k\}$. This in turn implies that $I_\omega(x_0, \dots, x_k) = 0$. Thus, the integrands on the right hand side of \eqref{eq:reduction_to_U-set} are identically zero for $\theta$ close enough to 1, which completes the proof of the claimed $\norm{I_\omega}_{L^p M^k(\Omega)}^p - [I_\omega]_{L^p M^k(V)}^p \le 0$.
\end{proof}

We then prove the next case of Theorem \ref{thm:equiv_of_norms_general}.

\begin{prop}\label{prop:equiv_case_unbdd_alt}
	Let $\Omega \subset \R^n$ be open, let $R > 0$, let $c > 0$, and let $\omega \in L^p(\wedge^k T^* \Omega)$, where we extend $\omega$ to $L^p(\wedge^k T^* \R^n)$ by setting $\omega = 0$ outside $\Omega$. Then 
	\[
		[I_\omega]_{L^p M^{k, c}(\Omega\mid R)} 
		=\norm{I_\omega}_{L^p M^{k, c}(\Omega\mid R)} 
		= K \bigl\lVert\abs{\omega}_{\S,p}\bigr\rVert_{L^p(\Omega)},
	\]
	where $K = K(p, k)$ and $\abs{\cdot}_{\S,p}$ are as in Proposition \ref{prop:equiv_case_unbdd_regular}.
\end{prop}

\begin{proof}
	By Proposition \ref{prop:equiv_case_unbdd_regular}, it suffices to show that
	\[
		[I_\omega]_{L^p M^{k, c}(\Omega\mid R)} 
		=\norm{I_\omega}_{L^p M^{k, c}(\Omega\mid R)} 
		=\norm{I_\omega}_{L^p M^{k}(\Omega\mid R)}.
	\]
	We first show that we may yet again assume that $\omega \in C_0(\wedge^k T^* \omega)$. Indeed, given a sequence of $\omega_j \in C_0(\wedge^k T^* \Omega)$ that converges to $\omega$ in $L^p(\wedge^k T^* \Omega)$, if the statement of the theorem applies to $\omega_j$, then we can apply Lemma \ref{lem:limit_switching_lemma} twice, once with $\norm{\cdot}_{L^p M^{k, c}(\Omega\mid R)}$ and once with $\norm{\cdot}_{L^p M^{k}(\Omega\mid R)}$, to conclude that
	\begin{multline*}
		\lim_{\theta \to 1^{-}} \smallnorm{I_{\omega}}_{L^p M^{k,c}(\Omega\mid R, \theta)}
		= \lim_{j \to \infty} \lim_{\theta \to 1^{-}} \smallnorm{I_{\omega_j}}_{L^p M^{k,c}(\Omega\mid R, \theta)}\\
		= \lim_{j \to \infty} \lim_{\theta \to 1^{-}} \smallnorm{I_{\omega_j}}_{L^p M^{k}(\Omega\mid R, \theta)}
		= \smallnorm{I_\omega}_{L^p M^{k}(\Omega\mid R)}.
	\end{multline*}
	Thus, we proceed to assume that $\omega \in C_0(\wedge^k T^* \omega)$. Since we trivially have $\norm{I_\omega}_{L^p M^{k,c}(\Omega\mid R)} \le \norm{I_\omega}_{L^p M^{k}(\Omega\mid R)}$, it suffices to show that $[I_\omega]_{L^p M^{k,c}(\Omega\mid R)} \ge \norm{I_\omega}_{L^p M^{k}(\Omega\mid R)}$. 
	
	For every $\delta > 0$, we define
	\[
		\Omega_{\delta} = \{x \in \Omega : \dist (x, \partial \Omega) > \delta\},
	\]
	where we interpret $\Omega_{\delta} = \Omega$ if $\partial \Omega = \emptyset$ (i.e.\ if $\Omega = \R^n$). Then, since $\omega \in C_0(\wedge^k T^* \omega)$, we may fix $\delta > 0$ to be such that $\spt \omega \subset \Omega_{\delta}$. By Lemma \ref{lem:compact_support_lemma}, we thus have
	\[
		\norm{I_{\omega}}_{L^p M^k(\Omega\mid R)} = [I_{\omega}]_{L^p M^k(\Omega_\delta\mid R)}.
	\]
	In particular, the claim follows if $[I_\omega]_{L^p M^{k,c}(\Omega\mid R)} \ge [I_{\omega}]_{L^p M^k(\Omega_\delta\mid R)}$.
	
	For this, note that if $(x_0, \dots, x_k) \in \B(\Omega_\delta, k, R) \setminus (\B(\Omega, k, R) \cap \Omega(k, c)) = \B(\Omega_\delta, k, R) \setminus \Omega(k, c)$, then there exists an index $i \in \{1, \dots, k\}$ such that $\abs{x_i - x_0} \ge c \dist(x_0, \partial \Omega) > c\delta$. Hence, we conclude that if $\theta$ is close enough to 1 that $R\eps_\theta < c\delta$,
	then $\B(\Omega_\delta, k, R) \setminus \Omega(k, c) \subset \B(\Omega, k, R) \setminus B(\Omega, k, R\eps_\theta)$. Thus, by Lemma \ref{lem:only_close_pts_matter}, we have
	\[
		\lim_{\theta \to 1^{-}} \int_{\B(\Omega_\delta, k, R) \setminus \Omega(k, c)} \frac{(1 - \theta)^k \abs{I_\omega(x_0, \dots, x_k)}^p}{\left(\abs{x_1 - x_0} \cdots \abs{x_k - x_0}\right)^{n + \theta p}} \, dx_0 \dots dx_k = 0.
	\]
	Consequently, we have
	\begin{multline*}
		[I_\omega]^p_{L^p M^{k}(\Omega_\delta \mid R)}\\
		= \liminf_{\theta \to 1^{-}} \int_{\B(\Omega_\delta, k, R) \cap \Omega(k, c)} \frac{(1 - \theta)^k \abs{I_\omega(x_0, \dots, x_k)}^p}{\left(\abs{x_1 - x_0} \cdots \abs{x_k - x_0}\right)^{n + \theta p}} \, dx_0 \dots dx_k\\
		\le [I_\omega]^p_{L^p M^{k, c}(\Omega\mid R)},
	\end{multline*}
	completing the proof.
\end{proof}

The last case of Theorem \ref{thm:equiv_of_norms_general} is then an immediate corollary.

\begin{cor}\label{cor:equiv_case_bdd_alt}
	Let $\Omega \subset \R^n$ be open and bounded, let $c > 0$, and let $\omega \in L^p(\wedge^k T^* \Omega)$, where we extend $\omega$ to $L^p(\wedge^k T^* \R^n)$ by setting $\omega = 0$ outside $\Omega$. Then 
	\[
		[I_\omega]_{L^p M^{k, c}(\Omega)} 
		=\norm{I_\omega}_{L^p M^{k, c}(\Omega)} 
		= K \bigl\lVert\abs{\omega}_{\S,p}\bigr\rVert_{L^p(\Omega)},
	\]
	where $K = K(p, k)$ and $\abs{\cdot}_{\S,p}$ are as in Proposition \ref{prop:equiv_case_unbdd_regular}.
\end{cor}
\begin{proof}
	The claim is equivalent to Proposition \ref{prop:equiv_case_unbdd_alt} with $R = \diam \Omega$, due to \eqref{eq:bdd_set_equality_alt}.
\end{proof}

By combining Propositions \ref{prop:equiv_case_unbdd_regular} and \ref{prop:equiv_case_unbdd_alt} with Corollaries \ref{cor:equiv_case_bdd_regular} and \ref{cor:equiv_case_bdd_alt}, the proof of Theorem \ref{thm:equiv_of_norms_general} is hence complete.

\section{Weak derivatives}

Let $\Omega$ be an open set in $\R^n$. As stated in the introduction, the \emph{Alexander-Spanier differential} $d \colon M^k(\Omega; \R) \to M^{k+1}(\Omega; \R)$ on measurable $k$-multi\-functions is defined by
\[
	(dF)(x_0, x_1, \dots, x_{k+1}) = \sum_{i = 0}^k (-1)^i F(x_0, x_1, \dots, \hat{x_i}, \dots, x_{k+1})
\]
for $F \in M^k(\Omega; \R)$. It is clear from the definition that $dF \colon \Omega^{k+2} \to \R$ is measurable if $F \colon \Omega^{k+1} \to \R$ is measurable, and that the map $d$ is linear. It is also well known that this map satisfies
\begin{equation}\label{eq:dd_is_zero}
	d(dF) = 0 \qquad \text{for all } F \in M^{k}(\Omega; \R), k \in \Z_{\geq 0}.
\end{equation}

We begin our work towards Theorem \ref{thm:BBM_for_Sobolev_forms_all_cases} by noting the following fact for smooth differential forms.

\begin{lemma}\label{lem:AS_Stokes_smooth}
	Let $\Omega \subset \R^n$ be open, and let $\omega \in C^\infty(\wedge^k T^* \Omega)$. Then for all $(x_0, \dots, x_{k+1}) \in \Omega^{k+2}$ such that $\Delta(x_0, \dots, x_{k+1}) \subset \Omega$, the Alexander-Spanier differential of $I_\omega$ satisfies
	\[
		dI_\omega(x_0, \dots, x_{k+1}) = I_{d\omega}(x_0, \dots, x_{k+1}).
	\]
	In particular, if $\Omega$ is convex, then this holds for all $(x_0, \dots, x_{k+1}) \in \Omega^{k+2}$.
\end{lemma}
\begin{proof}
	The boundary of a simplex $\Delta(x_0, \dots, x_{k+1})$ is given by the simplicial $k$-chain $\sum_{i=0}^\infty (-1)^i \Delta(x_0, \dots \hat{x_i}, \dots, x_{k+1})$. Thus, we have by Stokes' theorem that
	\begin{multline*}
		I_{d\omega}(x_0, \dots, x_{k+1}) = \int_{\Delta(x_0, \dots, x_{k+1})} d\omega
		= \int_{\partial \Delta(x_0, \dots, x_{k+1})} \omega\\
		= \sum_{i=0}^\infty (-1)^i \int_{\Delta(x_0, \dots \hat{x_i}, \dots, x_{k+1})} \omega = dI_\omega(x_0, \dots, x_{k+1}).
	\end{multline*}
\end{proof}

A smooth approximation argument then yields the following corollary; see also e.g.\ \cite[Corollary 7.1]{Kangasniemi-Prywes_SurfaceModulus}.

\begin{lemma}\label{lem:AS_Stokes}
	Let $\Omega \subset \R^n$ be open, and let $\omega \in W^{d, p}(\wedge^k T^* \Omega)$. Then for a.e.\ $(x_0, \dots, x_{k+1}) \in \Omega^{k+2}$ such that $\Delta(x_0, \dots, x_{k+1}) \subset \Omega$, the Alexander-Spanier differential of $I_\omega$ satisfies
	\[
		dI_\omega(x_0, \dots, x_{k+1}) = I_{d\omega}(x_0, \dots, x_{k+1}).
	\]
	In particular, if $\Omega$ is convex, then this holds for a.e.\ $(x_0, \dots, x_{k+1}) \in \Omega^{k+2}$.
\end{lemma}
\begin{proof}
	The set $C^\infty(\wedge^k T^* \Omega)$ is dense in $W^{d,p}(\wedge^k T^* \Omega)$. Thus, we may select a sequence of smooth forms $\omega_j \in C^\infty(\wedge^k T^* \Omega)$ such that $\norm{\omega - \omega_j}_{L^p(\Omega)} \to 0$ and $\norm{d\omega - d\omega_j}_{L^p(\Omega)} \to 0$ as $j \to \infty$. It follows from the so-called Fuglede lemma, see \cite[Theorem 3 (f)]{Fuglede_surface-modulus} and also e.g.\ \cite[Lemma 4.1]{Kangasniemi-Prywes_SurfaceModulus}, that by replacing $\omega_j$ with a subsequence, we may assume that 
	\[
		\int_{\Delta(x_0, \dots, x_k)} \omega_j \xrightarrow[n \to \infty]{} \int_{\Delta(x_0, \dots, x_k)} \omega
	\]
	outside a $p$-exceptional family of simplices $\Delta(x_0, \dots, x_k)$. The corresponding family $P \subset \Omega^{k+1}$ of exceptional $(k+1)$-tuples of points is thus a Lebesgue null-set by Lemma \ref{lem:modulus_to_measure}. 
	
	We then observe that the family $P' \subset \Omega^{k+2}$ corresponding to simplices with a face in $P$ is also a Lebesgue null-set: indeed, this follows since $P'$ is a finite union of coordinate-permuted versions of $\Omega \times P$. We conclude that
	\[
		dI_{\omega_j}(x_0, \dots, x_{k+1}) \xrightarrow[n \to \infty]{} dI_{\omega}(x_0, \dots x_{k+1})
	\]
	for all $(x_0, \dots, x_{k+1})$ outside $P'$. 
	
	Finally, a second use of the Fuglede lemma for $d\omega_j$ along with another use of Lemma \ref{lem:modulus_to_measure} shows that, after replacing $\omega_j$ with a further subsequence, we have
	\[
		I_{d\omega_j}(x_0, \dots, x_{k+1}) \xrightarrow[n \to \infty]{} I_{d\omega}(x_0, \dots x_{k+1})
	\]
	outside a Lebesgue null-set $Q \subset \Omega^{k+2}$. The claim follows since for all $(x_0, \dots, x_{k+1}) \in \Omega^{k+2}$ such that $\Delta(x_0, \dots, x_{k+1}) \subset \Omega$, we have $I_{d\omega_j} = dI_{\omega_n}$ for all $j$ by Lemma \ref{lem:AS_Stokes_smooth}.
\end{proof}

We also need the following $L^p$-estimate for $dI_\omega$ when $\omega$ is a convolution. The proof is along the same lines as most standard convolution estimates for $L^p$-norms.
\begin{lemma}\label{lem:convolution_estimate}
	 Let $\omega \in L^p(\wedge^k T^* \R^n)$, let $R > 0$, and let $\eta \in C^\infty_0(\R^n, [0, \infty))$ be a mollifying kernel, where $\norm{\eta}_{L^1(\R^n)} = 1$ and $\spt \eta \subset \B^n(0, \eps)$ for a given $\eps > 0$. Consider the convolved differential form $\eta \ast \omega \in C^\infty(\wedge^k T^* \R^n)$ defined by
	\[
		(\eta \ast \omega)_x = \int_{\R^n} \eta(y) \omega_{x-y} \, dy = \int_{\R^n} \eta(x-y) \omega_{y} \, dy.
	\]
	Then for all measurable sets $E_1, E_2 \subset \R^n$ with $\B^n(E_1, \eps) \subset E_2$, and for every $\theta \in (0, 1)$, we have
	\[
		\norm{dI_{\eta \ast \omega}}_{L^p M^{k+1}(E_1\mid R, \theta)} 
		\leq \norm{dI_{\omega}}_{L^p M^{k+1}(E_2\mid R, \theta)}.
	\]
\end{lemma}
\begin{proof}
	We may assume that $\omega$ is Borel, as changing $\omega$ in a null-set does not effect the convolution. Observe that for all $(x_0, \dots, x_k) \in (\R^n)^{k+1}$, we have
	\begin{multline*}
		\int_{\Delta_k} \int_{\R^n} \abs{ \eta(y)
		\omega_{x_0 + s_1(x_1 - x_0) + \dots + s_k(x_k - x_0) - y} (x_1 - x_0, \dots, x_k - x_0)}  \, dy ds\\
		\le m_k(\Delta_k) \abs{x_1 - x_0} \cdots \abs{x_k - x_0} \norm{\eta}_{L^\infty(\R^n)} \norm{\omega}_{L^p(\R^n)}^p < \infty.
	\end{multline*} 
	This allows us to use Fubini's theorem, which yields for a.e.\ $(x_0, \dots, x_{k}) \in (\R^n)^{k+1}$ that
	\begin{multline*}
		I_{\eta \ast \omega}(x_0, \dots, x_{k+1})
		= \int_{\Delta(x_0, \dots, x_{k+1})}
			\left( \int_{\R^n} \eta(y) \omega_{\cdot-y} \, dy \right)\\
		= \int_{\Delta_k} \int_{\R^n} \eta(y)
			\omega_{x_0 + s_1(x_1 - x_0) + \dots + s_k(x_k - x_0) - y} (x_1 - x_0, \dots, x_k - x_0)  \, dy ds\\
		= \int_{\R^n} \eta(y) I_\omega(x_0 - y, \dots, x_k - y) \, dy.
	\end{multline*}
	Thus, by the definition of the Alexander-Spanier differential, we have for a.e.\ $(x_0, \dots, x_{k+1}) \in (\R^n)^{k+2}$ that
	\[
		dI_{\eta \ast \omega}(x_0, \dots, x_{k+1})
		= \int_{\R^n} \eta(y) dI_\omega(x_0 - y, \dots, x_{k+1} - y) \, dy.
	\]
	Now, by using H\"older's inequality, we estimate
	\begin{multline*}
		\abs{dI_{\eta \ast \omega}(x_0, \dots, x_{k+1})}^p\\
		= \abs{\int_{\R^n} \eta^{\frac{p-1}{p}}(y) \cdot \eta^{\frac{1}{p}}(y) dI_\omega(x_0 - y, \dots, x_{k+1} - y) \, dy}^p\\
		\leq \norm{\eta}_{L^1(\R^n)}^{p-1} \int_{\R^n} \eta(y) \abs{dI_\omega(x_0 - y, \dots, x_{k+1} - y)}^p \, dy
	\end{multline*}
	for a.e.\ $(x_0, \dots, x_{k+1}) \in (\R^n)^{k+2}$. Since $\norm{\eta}_{L^1(\R^n)} = 1$, we can thus use Fubini-Tonelli to conclude that
	\begin{align*}
		&\norm{dI_{\eta \ast \omega}}_{L^p M^{k+1}(E_1\mid R, \theta)}^p\\
		&\quad\leq \int_{\B(E, k+1, R)} \int_{\R^n} \eta(y) (1-\theta)^{k+1} \\
		&\quad\hspace{3cm} \frac{\abs{dI_\omega(x_0 - y, \dots, x_{k+1} - y)}^p}{(\abs{x_1-x_0}\cdots\abs{x_{k+1} - x_0})^{n+\theta p}} \, dy dx_0 \dots dx_{k+1}\\
		&\quad=\int_{\R^n} \eta(y) \norm{d I_\omega}_{L^p M^{k+1}((E_1 - y)\mid R, \theta)}^p \, dy\\
		&\quad\leq \norm{\eta}_{L^1(\R^n)} \norm{d I_\omega}_{L^p M^{k+1}(E_2\mid R, \theta)}^p,
	\end{align*}
	completing the proof.
\end{proof}

We are then ready to prove the first case of Theorem \ref{thm:BBM_for_Sobolev_forms_all_cases}, namely case \eqref{enum:BBM_case_conv}.

\begin{prop}\label{prop:BBM_case_conv}
	Let $\Omega \subset \R^n$ be open and convex, $R > 0$, $p, \in (1, \infty)$, $k \in \{0, \dots, n-1\}$, and let $\omega \in L^p(\wedge^k T^* \Omega)$, where we extend $\omega$ to $L^p(\wedge^k T^* \R^n)$ by setting $\omega = 0$ outside $\Omega$. 
	Then the following conditions are equivalent:
	\begin{enumerate}[label=(\roman*)]
		\item $\omega \in W^{d,p}(\wedge^{k} T^*\Omega)$; \label{enum:BBM_conv_sob}
		\item $\norm{dI_\omega}_{L^p M^{k+1}(\Omega \mid R)} < \infty$; \label{enum:BBM_conv_sup}
		\item $[dI_\omega]_{L^p M^{k+1}(\Omega \mid R)} < \infty$. \label{enum:BBM_conv_inf}
	\end{enumerate}
	Moreover, if $\omega \in W^{d,p}(\wedge^{k} T^*\Omega)$, then 
	\[
		[dI_\omega]_{L^p M^{k+1}(\Omega \mid R)} 
		= \norm{dI_\omega}_{L^p M^{k+1}(\Omega \mid R)} 
		= K \bigl\lVert\abs{d\omega}_{\S,p}\bigr\rVert_{L^p(\Omega)},
	\]
	where $K = K(p, n, k)$ and $\abs{\cdot}_{\S,p}$ are as in Theorem \ref{thm:equiv_of_norms_general}.
\end{prop}

\begin{proof}
	It is clear that \ref{enum:BBM_conv_sup} implies \ref{enum:BBM_conv_inf}. We then prove that \ref{enum:BBM_conv_sob} implies \ref{enum:BBM_conv_sup}, along with the additional claims at the end of the theorem. For this, suppose that $\omega \in W^{d,p}(\wedge^k T^*\Omega)$. Then by Lemma \ref{lem:AS_Stokes}, we have $dI_\omega = I_{d\omega}$ a.e.\ on $\Omega^{k+2}$, and by Proposition \ref{prop:equiv_case_unbdd_regular}, we have
	\[
		[I_{d\omega}]_{L^p M^k(\Omega\mid R)}
		= \norm{I_{d\omega}}_{L^p M^k(\Omega\mid R)} 
		= K(n, p, k) \bigl\lVert\abs{d\omega}_{\S,p}\bigr\rVert_{L^p(\Omega)},
	\]
	thus proving the claim in this case.
	
	It remains to prove that \ref{enum:BBM_conv_inf} implies \ref{enum:BBM_conv_sob}. Thus, let $\omega \in L^p(\wedge^k T^*\Omega)$ with $[dI_{\omega}]_{L^p M^{k+1}(\Omega)} < \infty$, with the objective to show that $\omega \in W^{d,p}(\wedge^k T^* \Omega)$. For this, let $\eps > 0$, and let $\Omega_\eps = \{x \in \Omega : \dist (x, \partial \Omega) > \eps\} $. By Lemma \ref{lem:convex_approx}, $\Omega_\eps$ is a convex subdomain of $\Omega$ with $\B^n(\Omega_\eps, \eps) \subset \Omega$. 
	
	We zero extend $\omega$ outside $\Omega$ so that $\omega \in L^p(\wedge^k T^* \R^n)$, and select smooth approximations $\omega_j = \eta_j \ast \omega$ of $\omega$, where the sequence of mollifying kernels $\eta_j \in C^\infty_0(\R^n, [0, \infty])$ is chosen so that $\spt \eta_j \subset \B^n(0, \eps)$ for all $j \in \Z_{> 0}$. In particular, by component-wise application of standard properties of convolution approximations, $\omega_j$ are smooth $k$-forms with $\lim_{j \to \infty} \norm{\omega_j - \omega}_{L^p(\R^n)} = 0$.
	
	We fix an index $j$. By Proposition \ref{prop:equiv_case_unbdd_regular}, we have that $\norm{d\omega_j}_{L^p(\Omega_\eps)} \lesssim_{n,p,k} [I_{d\omega_j\vert_{\Omega_\eps}}]_{L^p M^{k+1}(\Omega_\eps\mid R)}$. However, since $\Omega_\eps$ is convex, the values of $I_{d\omega_j}$ in $\Omega_\eps$ do not depend on how $d\omega_j$ is defined outside $\Omega_\eps$, and thus we in fact have $[I_{d\omega_j\vert_{\Omega_\eps}}]_{L^p M^{k+1}(\Omega_\eps\mid R)} = [I_{d\omega_j}]_{L^p M^{k+1}(\Omega_\eps\mid R)}$.
	Moreover, since $\omega_j$ are smooth, we have by Lemma \ref{lem:AS_Stokes_smooth} that $dI_{\omega_j} = I_{d\omega_j}$ in all of $\R^n$. Lastly, by Lemma \ref{lem:convolution_estimate} and the condition $\spt \eta_j \subset \B^n(0, \eps)$, we know that $[dI_{\omega_j}]_{L^p M^{k+1}(\Omega_\eps\mid R)} \le [dI_{\omega}]_{L^p M^{k+1}(\Omega\mid R)}$. By chaining all of the above deductions together, we hence have the uniform upper bound
	\[
		\norm{d\omega_j}_{L^p(\Omega_\eps)} 
		\le [dI_{\omega}]_{L^p M^{k+1}(\Omega\mid R)} < \infty
	\]
	for all indices $j$.
	
	Since $p > 1$, $L^p(\wedge^{k+1} T^* \Omega_\eps)$ is reflexive, and thus balls in it are weakly compact by the by the Banach-Alaoglu theorem. Thus, by passing to a subsequence of $\omega_j$, we may assume that there exists a $\tau \in L^p(\wedge^{k+1} T^* \Omega_\eps)$ such that $d\omega_j \rightharpoonup \tau$ weakly in $L^p(\wedge^{k+1} T^* \Omega_\eps)$. However, now if $\eta \in C^\infty_0(\wedge^{n-k-1} T^* \Omega_\eps)$ is a smooth test form, we can use the smoothness of $\omega_j$ to conclude that
	\begin{align*}
		&\abs{\int_{\Omega_\eps} \left(\tau \wedge \eta - (-1)^{k+1} \omega \wedge d\eta \right)}\\
		&\qquad = \abs{\int_{\Omega_\eps} \left((\tau - d\omega_j) \wedge \eta - (-1)^{k+1} (\omega - \omega_j) \wedge d\eta \right)}\\
		&\qquad \le \abs{\int_{\Omega_\eps} d\omega_j \wedge \eta - \int_{\Omega_\eps} \tau \wedge \eta}
		+ \norm{d\eta}_{L^{p/(p-1)}(\R^n)} \norm{\omega - \omega_j}_{L^{p}(\R^n)}.
	\end{align*}
	Since $d\omega_j \rightharpoonup \tau$ weakly in $L^p(\wedge^{k+1} T^* \Omega_\eps)$ and $\omega_j \to \omega$ strongly in $L^p(\wedge^k T^*\R^n)$, the right hand side of the above estimate tends to zero as $j \to \infty$. Thus,  $\tau$ is a weak exterior derivative of $\omega$ in $\Omega_\eps$, and hence $\omega \in W^{d,p}(\wedge^{k} T^* \Omega_\eps)$. 
	
	We then apply the case that was proven at the beginning of the proof to obtain
	\begin{equation}\label{eq:eps_upper_bound}
		\norm{d\omega}_{L^p(\Omega_\eps)} \lesssim_{n,k,p} [dI_{\omega}]_{L^p M^k(\Omega_\eps\mid R)}
		\leq [dI_{\omega}]_{L^p M^k(\Omega\mid R)} < \infty,
	\end{equation}
	where we again leverage the convexity of $\Omega_\eps$ to avoid having to switch to a zero extension of $\omega$ in the first estimate. Finally we note that $\Omega = \bigcup_{\eps > 0} \Omega_\eps$ by openness of $\Omega$. Thus, $\omega \in W^{d,p}_\loc(\wedge^k T^*\Omega)$, and by applying monotone convergence on \eqref{eq:eps_upper_bound}, we conclude that $\norm{d\omega}_{L^p(\Omega)} \lesssim_{n,k,p} [dI_{\omega}]_{L^p M^k(\Omega\mid R)} < \infty$, completing the proof.
\end{proof}

As before, due to \eqref{eq:bdd_set_equality}, case \eqref{enum:BBM_case_bdd_conv} is an immediate corollary of Proposition \ref{prop:BBM_case_conv} with $R = \diam \Omega$. Notably, this case completes the proof of Theorem \ref{thm:BBM_for_forms_convex}.

\begin{cor}\label{cor:BBM_case_conv_bdd}
	Let $\Omega \subset \R^n$ be a bounded, convex domain, $p, \in (1, \infty)$, $k \in \{0, \dots, n-1\}$, and let $\omega \in L^p(\wedge^k T^* \Omega)$, where we extend $\omega$ to $L^p(\wedge^k T^* \R^n)$ by setting $\omega = 0$ outside $\Omega$. 
	Then the following conditions are equivalent:
	\begin{enumerate}[label=(\roman*)]
		\item $\omega \in W^{d,p}(\wedge^{k} T^*\Omega)$; 
		\item $\norm{dI_\omega}_{L^p M^{k+1}(\Omega)} < \infty$;
		\item $[dI_\omega]_{L^p M^{k+1}(\Omega)} < \infty$. 
	\end{enumerate}
	Moreover, if $\omega \in W^{d,p}(\wedge^{k} T^*\Omega)$, then 
	\[
		[dI_\omega]_{L^p M^{k+1}(\Omega)} 
		= \norm{dI_\omega}_{L^p M^{k+1}(\Omega)} 
		= K \bigl\lVert\abs{d\omega}_{\S,p}\bigr\rVert_{L^p(\Omega)},
	\]
	where $K = K(p, n, k)$ and $\abs{\cdot}_{\S,p}$ are as in Theorem \ref{thm:equiv_of_norms_general}.
\end{cor}

We then prove the parts of Theorem \ref{thm:BBM_for_Sobolev_forms_all_cases} with non-convex $\Omega$. We again start with the unbounded version.

\begin{prop}\label{prop:BBM_case_arbitrary}
	Let $\Omega \subset \R^n$ be open, $R > 0$, $c \in (0, 1]$, $p, \in (1, \infty)$, $k \in \{0, \dots, n-1\}$, and let $\omega \in L^p(\wedge^k T^* \Omega)$, where we extend $\omega$ to $L^p(\wedge^k T^* \R^n)$ by setting $\omega = 0$ outside $\Omega$. 
	Then the following conditions are equivalent:
	\begin{enumerate}[label=(\roman*)]
		\item $\omega \in W^{d,p}(\wedge^{k} T^*\Omega)$; \label{enum:BBM_gen_sob}
		\item $\norm{dI_\omega}_{L^p M^{k+1,c}(\Omega \mid R)} < \infty$; \label{enum:BBM_gen_sup}
		\item $[dI_\omega]_{L^p M^{k+1,c}(\Omega \mid R)} < \infty$. \label{enum:BBM_gen_inf}
	\end{enumerate}
	Moreover, if $\omega \in W^{d,p}(\wedge^{k} T^*\Omega)$, then 
	\[
		[dI_\omega]_{L^p M^{k+1,c}(\Omega \mid R)} 
		= \norm{dI_\omega}_{L^p M^{k+1,c}(\Omega \mid R)} 
		= K \bigl\lVert\abs{d\omega}_{\S,p}\bigr\rVert_{L^p(\Omega)},
	\]
	where $K = K(p, n, k)$ and $\abs{\cdot}_{\S,p}$ are as in Theorem \ref{thm:equiv_of_norms_general}.
\end{prop}

\begin{proof}
	We again trivially have that \ref{enum:BBM_gen_sup} implies \ref{enum:BBM_gen_inf}. We start by showing that \ref{enum:BBM_gen_inf} implies \ref{enum:BBM_gen_sob}. Thus, suppose that $\omega \in L^p(\wedge^k T^* \Omega)$ with  $[dI_\omega]_{L^p M^{k+1,c}(\Omega\mid R)} < \infty$. Let $x \in \Omega$, and let $B$ be a ball of the form
	\[
		B = \B^n(x, r), \quad \text{where } 0 < r < \min \left( \frac{c}{4}, \frac{1}{2}, \frac{R}{2} \right) \dist(x, \partial \Omega).
	\] 
	Note that since $c > 0$, we have $r \in (0, 1]$. Thus, $B$ is a neighborhood of $x$, and $B \subset \Omega$. 
	
	We first claim that $B^{k+2} \subset \B(\Omega, k+1, R) \cap \Omega(k+1, c)$. For this, let $(x_0, \dots, x_{k+1}) \in B^{k+2}$. For every $i \in \{1, \dots, k+1\}$, we have
	\[
		\abs{x_0 - x_i} \le \abs{x_0 - x} + \abs{x - x_i} < \min \left(R,  \frac{c}{2} \dist(x, \partial \Omega)\right).
	\] 
	In particular, $B^{k+2} \subset \B(\Omega, k+1, R)$. In addition to this, we have 
	\[
		\dist(x_0, \partial \Omega) \ge \dist(x, \partial \Omega) - \abs{x - x_0} >
		\frac{1}{2}\dist(x, \partial \Omega).
	\]
	Thus, $\abs{x_0 - x_i} < c \dist(x_0, \partial \Omega)$ for all $i \in \{1, \dots, k\}$, completing the proof that  $B^{k+2} \subset \B(\Omega, k+1, R) \cap \Omega(k+1, c)$.
	
	Now, we have
	\[
		\norm{dI_\omega}_{L^p M^{k+1}(B)} \le \norm{dI_\omega}_{L^p M^{k+1, c}(\Omega\mid R)} < \infty.
	\]
	Thus, by Corollary \ref{cor:BBM_case_conv_bdd} and the convexity of $B$, we have $\omega \in W^{d,p}(\wedge^k T^* B_x)$. Since the sets $B$ as above form an open cover of $\Omega$, it follows that $\omega \in W^{d,p}_\loc(\wedge^k T^* \Omega)$. 
	
	Moreover, the sets $B$ form a Vitali covering of $\Omega$, and we may thus use the Vitali covering theorem to select a countable pairwise disjoint subcollection $B_i$ that covers $\Omega$ up to a nullset. By Proposition \ref{prop:BBM_case_conv}, we conclude that
	\begin{equation*}
		\norm{d\omega}_{L^p(\Omega)}^p 
		= \sum_i \norm{d\omega}_{L^p(B_i)}^p
		\lesssim_{p,n,k} \sum_i [I_{d\omega}]_{L^p M^{k+1}(B_i)}^p
	\end{equation*}
	Now, since $\liminf$ is superadditive, and since the balls $B_i$ are disjoint, we have 
	\begin{multline*}
		\sum_i [I_{d\omega}]_{L^p M^{k+1}(B_i)}^p
		\le \liminf_{\theta \to 1^{-}} \sum_i \norm{I_{d\omega}}_{L^p M^{k+1}(B_i, \theta)}^p\\
		= \liminf_{\theta \to 1^{-}} \int_{\bigcup_i B_i^{k+2}} \frac{(1 - \theta)^k \abs{I_{d\omega}(x_0, \dots, x_{k+1})}^p}{\left(\abs{x_1 - x_0} \cdots \abs{x_{k+1} - x_0}\right)^{n + \theta p}} \, dx_0 \dots dx_{k+1}.
	\end{multline*}
	Now, since all $B_i$ are convex subsets of $\Omega$, Lemma \ref{lem:AS_Stokes} applies on $\bigcup_i B_i^{k+2}$. By combining this with the fact that $B_i^{k+2} \subset \B(\Omega, k+1, R) \cap \Omega(k+1, c)$ for all $i$, we obtain
	\begin{multline*}
		\liminf_{\theta \to 1^{-}} \int_{\bigcup_i B_i^{k+2}} \frac{(1 - \theta)^k \abs{I_{d\omega}(x_0, \dots, x_{k+1})}^p}{\left(\abs{x_1 - x_0} \cdots \abs{x_{k+1} - x_0}\right)^{n + \theta p}} \, dx_0 \dots dx_{k+1}\\
		= \liminf_{\theta \to 1^{-}} \int_{\bigcup_i B_i^{k+2}} \frac{(1 - \theta)^k \abs{dI_{\omega}(x_0, \dots, x_{k+1})}^p}{\left(\abs{x_1 - x_0} \cdots \abs{x_{k+1} - x_0}\right)^{n + \theta p}} \, dx_0 \dots dx_{k+1}\\
		\le [dI_{\omega}]_{L^p M^{k+1}(\Omega\mid R)}^p < \infty.
	\end{multline*}
	Thus, we have $\norm{d\omega}_{L^p(\Omega)}^p < \infty$, completing the proof that \ref{enum:BBM_gen_inf} implies \ref{enum:BBM_gen_sob}.
	
	It remains to prove that \ref{enum:BBM_gen_sob} implies \ref{enum:BBM_gen_sup}, along with the other claims at the end of the statement. Thus, suppose that $\omega \in W^{d,p}(\wedge^k T^* \Omega)$. Let $(x_0, \dots, x_{k+1}) \in \Omega(c, k)$. By the definition of $\Omega(c, k)$, all the coordinates $x_i$ are contained in $\B^n(x_0, c\dist(x_0, \partial \Omega))$. Since $\B^n(x_0, c\dist(x_0, \partial \Omega))$ is convex, it follows that $\Delta(x_0, \dots, x_{k+1}) \subset \B^n(x_0, c\dist(x_0, \partial \Omega))$. Moreover, since $c \le 1$, we have $\B^n(x_0, c\dist(x_0, \partial \Omega)) \subset \Omega$, and therefore $\Delta(x_0, \dots, x_{k+1}) \subset \Omega$. 
	
	Thus, Lemma \ref{lem:AS_Stokes} applies in all of $\Omega(c, k)$. We hence have $dI_\omega = I_{d\omega}$ a.e.\ in $\Omega(c, k)$, and thus also a.e.\ in $\B(\Omega, k+1, R) \cap \Omega(c, k)$. Now, it remains to apply Proposition \ref{prop:equiv_case_unbdd_alt}, with which we obtain
	\[
		\norm{dI_\omega}_{L^p M^{k+1,c}(\Omega\mid R)}
		= \norm{I_{d\omega}}_{L^p M^{k+1,c}(\Omega\mid R)}
		= K \bigl\lVert\abs{d\omega}_{\S,p}\bigr\rVert_{L^p(\Omega)} < \infty.
	\]
	The same application of Lemma \ref{lem:AS_Stokes} and Proposition \ref{prop:equiv_case_unbdd_alt} also yields
	\[
		[dI_\omega]_{L^p M^{k+1,c}(\Omega\mid R)}
		= [I_{d\omega}]_{L^p M^{k+1,c}(\Omega\mid R)}
		= K \bigl\lVert\abs{d\omega}_{\S,p}\bigr\rVert_{L^p(\Omega)},
	\]
	completing the proof.
\end{proof}

We again get the last case of Theorem \ref{thm:BBM_for_Sobolev_forms_all_cases} as a trivial consequence of the previous proposition with $R = \diam \Omega$ and \eqref{eq:bdd_set_equality_alt}.

\begin{prop}\label{cor:BBM_case_bdd}
	Let $\Omega \subset \R^n$ be open and bounded, $c \in (0, 1]$, $p, \in (1, \infty)$, $k \in \{0, \dots, n-1\}$, and let $\omega \in L^p(\wedge^k T^* \Omega)$, where we extend $\omega$ to $L^p(\wedge^k T^* \R^n)$ by setting $\omega = 0$ outside $\Omega$. 
	Then the following conditions are equivalent:
	\begin{enumerate}[label=(\roman*)]
		\item $\omega \in W^{d,p}(\wedge^{k} T^*\Omega)$; 
		\item $\norm{dI_\omega}_{L^p M^{k+1,c}(\Omega)} < \infty$; 
		\item $[dI_\omega]_{L^p M^{k+1,c}(\Omega)} < \infty$. 
	\end{enumerate}
	Moreover, if $\omega \in W^{d,p}(\wedge^{k} T^*\Omega)$, then 
	\[
		[dI_\omega]_{L^p M^{k+1,c}(\Omega)} 
		= \norm{dI_\omega}_{L^p M^{k+1,c}(\Omega)} 
		= K \bigl\lVert\abs{d\omega}_{\S,p}\bigr\rVert_{L^p(\Omega)},
	\]
	where $K = K(p, n, k)$ and $\abs{\cdot}_{\S,p}$ are as in Theorem \ref{thm:equiv_of_norms_general}.
\end{prop}

Thus, by combining Propositions \ref{prop:BBM_case_conv} and \ref{prop:BBM_case_arbitrary} with Corollaries \ref{cor:BBM_case_conv_bdd} and \ref{cor:BBM_case_bdd}, the proof of Theorem \ref{thm:BBM_for_Sobolev_forms_all_cases} is complete.

It remains to prove Theorem \ref{thm:BBM_for_Sobolev_forms_nonconvex}, which is a relatively straightforward corollary of the results shown so far. We first recall the statement for convenience of the reader.

\BBMFormsNonconv*

\begin{proof}
	For \ref{enum:BBM_to_Sobolev}, if $[dI_\omega]_{L^p M^{k+1}(\Omega)}$ is finite, then we have
	\[
		[dI_\omega]_{L^p M^{k+1, c}(\Omega)} \le [dI_\omega]_{L^p M^{k+1}(\Omega)} < \infty.
	\]
	Thus, the claim follows immediately from Theorem \ref{thm:BBM_for_Sobolev_forms_all_cases}.
	For \ref{enum:Sobolev_to_BBM}, let $\omega \in W^{d,p}(\wedge^k T^* \Omega)$, and let $\omega_{\ext} \in W^{d,p}(\wedge^k T^* B)$ be a $W^{d,p}$-extension of $\omega$. Since $B$ is convex, we may thus use Theorem \ref{thm:equiv_of_norms_general} and \ref{thm:BBM_for_Sobolev_forms_all_cases} to obtain that
	\[
		\norm{dI_{\omega_{\ext}}}_{L^p M^{k+1}(\Omega)}
		\le \norm{dI_{\omega_{\ext}}}_{L^p M^{k+1}(B)}
		\lesssim_{p,n,k} \norm{d\omega_{\ext}}_{L^p(B)}
		< \infty,
	\]
	completing the proof of \ref{enum:Sobolev_to_BBM}.
\end{proof}


\bibliographystyle{abbrv}
\bibliography{sources}

\end{document}